\documentclass[11pt]{amsart}

\usepackage{amssymb}
\usepackage{epsfig}
\usepackage{comment}
\usepackage{amsmath}
\usepackage{subcaption}
\usepackage[section]{placeins}
\usepackage{float}
\usepackage{mathrsfs}
\usepackage{graphicx}
\usepackage{tikz}
\usetikzlibrary{arrows.meta,calc,patterns,positioning}
\usepackage[notrig]{physics}
\usepackage{units}
\usepackage[T1]{fontenc}
\usepackage{lmodern}
\usepackage{amsthm,mathtools}
\usepackage{enumitem}
\usepackage{geometry}
\usepackage{esint}
\numberwithin{equation}{section}

\newtheorem{theorem}{Theorem}[section]
\newtheorem{lemma}[theorem]{Lemma}
\newtheorem{proposition}[theorem]{Proposition}
\newtheorem{definition}[theorem]{Definition}
\newtheorem{corollary}[theorem]{Corollary}
\newtheorem{remark}[theorem]{Remark}
\newtheorem{example}[theorem]{Example}

\newcommand{\mres}{\mathbin{\vrule height 1.6ex depth 0pt width
0.13ex\vrule height 0.13ex depth 0pt width 1.3ex}}

\begin{document}

\title[Microstructure evolution as a game]{Microstructure evolution as a game }

\author{M.~Ortiz}

\address
{
    Division of Engineering and Applied Science, 
    California Institute of Technology, 1200 E.~California Blvd., Pasadena, CA 91125 USA; 
    Centre Internacional de M\`etodes Num\`erics a l'Enginyeria (CIMNE),
    Universitat Polit\`ecnica de Catalunya, Jordi Girona 1, 08034 Barcelona, Spain.
}

\email{ortiz@caltech.edu}

\begin{abstract}
We develop a whole-trajectory relaxation framework for dissipative evolutions with non-convex, state- and history-dependent kinetics. The framework arises naturally from a game-theoretical formulation of energy--dissipation evolution. State--rate consistency is eliminated by causal Volterra reconstruction, reducing the problem to a diagonal equilibrium condition on complete rate trajectories. We establish a direct method based on vanishing-defect approximability, compactness, and complementary semicontinuity. The relaxation mechanism is computed exactly for a reduced simple-shear model with a non-convex kinetic potential and a local memory variable. The enriched state is a Young measure on complete local histories. We identify explicitly the sequential lower-semicontinuous envelope of the pure diagonal defect, characterize its zero set, construct pure recovery sequences for every relaxed zero, and prove existence through a frozen-best-response time discretization. The resulting coexistence of distinct rate populations with distinct memories provides a reduced mechanism for kinetic shear banding and connects naturally with negative strain-rate sensitivity and dynamic strain aging such as underlie the Portevin--Le Ch\^atelier effect.
\end{abstract}

\maketitle

\tableofcontents

{\small

\noindent\textbf{Keywords.}
Energy--dissipation evolution; non-convex kinetics; evolutionary relaxation;
diagonal equilibrium; history-dependent microstructure; Young measures;
recovery sequences; rate-dependent plasticity; kinetic shear banding.

\noindent\textbf{2020 Mathematics Subject Classification.}
49J45 (primary); 47J35, 74D10 (secondary).

}

\section{Introduction}
\label{sec:introduction}

Microstructure can be a manifestation of non-convex energy minimization, as in martensitic phase transformations and shape-memory alloys \cite{Bhattacharya2003}. The calculus of variations then provides a natural language of weak convergence, lower-semicontinuous envelopes, and recovery sequences. In many dissipative materials, however, non-convexity enters not only through the stored energy but through the kinetics themselves. Fine-scale patterns may therefore be selected dynamically rather than by a sequence of static energy minimizations. Representative examples include, to mention a few:

\begin{itemize}
\item[i)] single-crystal plasticity with latent hardening, where competition among slip systems produces non-convex and state/branch-dependent incremental problems, slip laminates, and dislocation microstructures \cite{OrtizRepetto1999, ContiOrtiz2005, ContiHauretOrtiz2007};
\item[ii)] jerky flow associated with dynamic strain aging, including the Portevin--Le Ch\^atelier phenomena, where negative strain-rate sensitivity drives propagating localization bands \cite{Yilmaz2011};
\item[iii)] shear-banding complex fluids, where non-monotone kinetics support coexisting bands with different shear rates and internal structures \cite{DivouxEtAl2016};
\item[iv)] adiabatic shear localization in dynamically loaded metals, where hardening competes with microstructural or thermal softening \cite{OsovskiRittelVenkert2013}; and
\item[v)] driven electrochemical systems, where electro-autocatalytic reaction kinetics can suppress or induce heterogeneous insertion patterns under nonequilibrium operation \cite{BaiCogswellBazant2011,Bazant2017}.
\end{itemize}

A particularly transparent manifestation of history-dependent kinetic microstructure is furnished by dynamic strain aging and the Portevin--Le Ch\^atelier effect \cite{Yilmaz2011,McCormick1988,Ananthakrishna2007}. In the instability regime, plastic flow localizes into deformation bands and the constitutive response may exhibit negative strain-rate sensitivity. Internal variables describing aging, pinning, or solute--dislocation interactions carry local memory, so populations that have experienced different rate histories need not have the same subsequent kinetic response even when their present macroscopic strain is the same. This gives a direct version of the restart obstruction: an instantaneous convexification of the rate law does not record which microscopic population acquired which internal history. The complete trajectory, rather than only the current relaxed rate, is therefore the natural object to compactify.

This distinction exposes a basic difficulty with incremental relaxation. Classical evolution theories provide several powerful but structurally different principles. Monotone-operator and variational-inequality methods treat coercive nonlinear responses \cite{Minty1962,LerayLions1965,Brezis1968}. Minimizing movements and curves of maximal slope construct gradient evolutions from time-incremental minimization \cite{DeGiorgi1993,Ambrosio1995,DeGiorgiMarinoTosques1980,AmbrosioGigliSavare2008}. Energetic and Balanced-Viscosity formulations describe rate-independent evolutions, including jumps \cite{MielkeTheil2004,MielkeRossiSavare2012,MielkeRossiSavare2016}. Parametrized quasi-static evolutions provide a related trajectory-level description in terms of stationarity and energy balance, together with constructive approximation and evolutionary $\Gamma$-convergence results; see, in particular, Negri \cite{Negri2014}. Trajectory-based minimum principles provide another route for dissipative systems \cite{ContiOrtiz2008,MielkeOrtiz2008,MielkeStefanelli2011}.

Whole-trajectory variational formulations are therefore not new in themselves, but the relaxation mechanism used here is different. In their classical form, the Br\'ezis--Ekeland principle and the related Nayroles minimum principles characterize convex dissipative evolutions by minimizing a nonnegative space--time functional built from a potential and its Fenchel conjugate \cite{BrezisEkeland1976,Nayroles1976}; convex duality is thus built into the construction. The weighted dissipation--energy family of Mielke and Ortiz \cite{MielkeOrtiz2008}, and the related weighted energy--dissipation formulation \cite{MielkeStefanelli2011}, likewise act on complete trajectories. In the Mielke--Ortiz construction, however, the direct $\Gamma$-limit as the regularization parameter tends to zero is highly degenerate and does not itself furnish a nontrivial relaxed trajectory functional; limiting rate-independent trajectories are instead identified through parameter-independent estimates and compactness. Conti and Ortiz \cite{ContiOrtiz2008} derived rigorous relaxations for two specific energy--dissipation functionals, but those model calculations do not supply a general closure principle for anomalous, state/branch-dependent kinetics.

The difficulty addressed here is precisely that such kinetics may fail to admit stable pure incremental minimizers. The endpoint of a time step may then be represented only by a minimizing sequence or generalized microstructure, so there is no pure state from which to start the next step; moreover, different microscopic realizations with the same relaxed endpoint can have different costs of subsequent rearrangement. This \emph{restart problem} \cite{ContiOrtiz2008} shows that the issue is not merely to compactify individual states or rates: one must compactify the \emph{entire state--rate evolution} while retaining the initial condition, causal history, work, and dissipation. The present construction therefore relaxes the whole-path equilibrium principle itself: the pure competitor class is retained, while relaxed solutions are generated as enriched-space limits of pure almost equilibria whose complete diagonal defects vanish.

The organizing idea of this paper is to make that whole-trajectory structure explicit as a game. At the heuristic level, each time $t$ may be regarded as a player choosing an instantaneous rate, with later players seeing the state accumulated from earlier choices. Passing from pointwise-in-time comparisons to an integral comparison produces a path-level game between state and rate histories: the state path enforces the consistency relation $\dot u=v$, while the rate path minimizes the energy--dissipation cost. For the class considered here, the decisive simplification is that the state player has a unique finite response and can be eliminated exactly through the Volterra operator:
\begin{equation}
Rv:=u_0+Kv,
\quad
J(w;v):=G(w;Rv) ,
\end{equation}
where $G(v;u)$ is an energy--dissipation functional that encodes the physics of the system. The first argument is the trial rate, while the second, written after the semicolon, generates the state. The evolution is then treated directly as a diagonal equilibrium
\begin{equation}
J(v;v)\leq J(w;v), 
\quad
\forall w.
\end{equation}
If the competitor section is proper, convex, and lower semicontinuous, the same condition can be written
\begin{equation}
0\in L(v),
\quad
L(v):=\partial_1J(v;v).
\end{equation}
This reduced formulation is closely related to variational inequalities, noncooperative dynamic games, and equilibrium problems \cite{Stackelberg1934,BasarOlsder1999,Fan1961,BlumOettli1994}, but its characteristic feature is the diagonal dependence induced by causal reconstruction. The exact Volterra reconstruction also explains why the game structure becomes comparatively unobtrusive in the examples treated below: the state player's best response is single-valued, explicit, and can be solved out. For more general causal constraints the state response may be available only implicitly, or exact elimination may be analytically undesirable; in that setting the unreduced state--rate game remains a substantive part of the formulation. The continuity equation in mass transport, discussed in the concluding remarks, is a prototype in which the causal coupling is a PDE constraint rather than an elementary Volterra operator.

Keeping the \emph{full} diagonal response is essential. In a quadratic Hilbert-space setting it produces the generally nonsymmetric bilinear form obtained after substituting $u=Rv$. If $S\in\mathcal L(H)$ is time independent, self-adjoint, and positive semidefinite, then the Volterra contribution satisfies
\begin{equation}
\int_0^T (SKv(t),v(t))_H\,dt
=
\frac12(SKv(T),Kv(T))_H
\geq 0.
\end{equation}
Indeed, $(Kv)'=v$ and self-adjointness gives $\frac{d}{dt}\frac12(SKv,Kv)_H=(SKv,v)_H$. This favorable causal term is lost if the state force is replaced by an arbitrary external perturbation. More generally, the two occurrences of the rate variable play different analytical roles: the first is the competitor and the second generates the state. Separating them reveals the complementary compactness and semicontinuity properties needed to pass to the diagonal limit.

The resulting direct method is naturally organized by the \emph{diagonal defect}
\begin{equation}
I(v):=\sup_{w\in Y}\bigl\{J(v;v)-J(w;v)\bigr\}.
\end{equation}
Zeros of $I$ are diagonal equilibria, while sequences with $I(v_h)\to0$ are almost equilibria of the complete path problem. The abstract existence theory developed below separates three logically distinct ingredients: vanishing-defect approximability, compactness of the corresponding almost equilibria, and lower/upper semicontinuity for the diagonal and competitor terms, respectively. The approximability requirement is substantive rather than cosmetic: coercivity and semicontinuity alone do not force the minimum defect to vanish. This formulation is close in spirit to the direct method for variational problems and to existence mechanisms for variational inequalities and convex integral functionals \cite{BrowderHess1972, LerayLions1965, Brezis1968, Zeidler1990, Showalter1997, Rockafellar1971IntegralII, Rockafellar1971ConvexIntegralDuality}, but here the direct method acts on a causal diagonal.

Two attainment results show how this architecture works in different growth regimes. For a representative superlinear dissipation model, the natural rate space is reflexive; the Volterra term supplies a favorable monotonicity contribution, the defect is coercive and weakly lower semicontinuous, and a temporal Galerkin construction converges strongly to the unique equilibrium. For a representative linear-growth model, by contrast, the natural ambient rate variable is a vector-valued Radon measure and the reconstructed state is a right-continuous $BV$ path. The work term then requires the centered state representative at atoms: this choice is singled out by the exact quadratic chain rule, so jumps need no \emph{ad hoc} energy correction. Energy balance yields a data-dependent total-variation bound and hence weak-star compactness. Viscous regularization provides an $L^2$ core, temporal measure-approximation spaces approximate entire trajectories, and the limiting measure equilibrium is unique; under the $W^{1,1}$ loading hypothesis used here, the compactness passage further shows that the equilibrium rate is absolutely continuous. Thus, measure-valued rates are not introduced merely as a formal relaxation but as the correct compact ambient variables at linear growth.

The same viewpoint also provides a direct relaxation method when exact equilibria are lost because the pure path space is not compact or the defect is not lower semicontinuous. Rather than postulating a new pointwise game on an enlarged space, we embed the pure problem in an enriched topology and define relaxed equilibria as limits of pure paths whose diagonal defects vanish. Equivalently, they are zeros of the sequential lower-semicontinuous envelope of the pure defect. This distinction matters: the enriched variables are chosen to retain the observables and costs that survive along recovery sequences, but the infinite extensions used to embed the pure payoff do not by themselves define meaningful best responses at genuinely relaxed points. The relaxation is therefore a closure of the \emph{evolutionary equilibrium principle}, not simply a pointwise convexification of its state variables.

The simple-shear model of Section~\ref{sec:local-memory-banding} makes this mechanism concrete in a setting where the relaxation can be solved completely. A non-convex kinetic potential permits coexistence of distinct local shear rates at a common stress, while a local internal variable evolves causally with the rate and therefore gives different bands different memories. The enriched variable is a Young measure on complete local histories. For this example the relaxed diagonal defect is computed exactly, every relaxed zero is generated by pure almost equilibria, and vanishing-defect approximability follows from an explicit frozen-best-response construction. The example also shows why replacing the kinetic potential by its pointwise convex envelope is generally insufficient when the internal-variable evolution depends nonlinearly on the microscopic rate.

The present paper develops these ideas in six steps:
\begin{itemize}
\item[i)] derive the integral evolution game and eliminate state--rate consistency through the reduced payoff $J(w;v)=G(w;Rv)$;
\item[ii)] establish abstract direct existence criteria based on vanishing-defect approximability, compactness, and complementary semicontinuity;
\item[iii)] prove existence, uniqueness, and temporal Galerkin convergence for a superlinear-growth energy--dissipation system;
\item[iv)] treat linear growth in vector-measure form, including the centered $BV$ reconstruction, energy-derived weak-star compactness, viscous regularization, and temporal measure approximation;
\item[v)] define relaxed diagonal equilibria as enriched-space limits of pure almost equilibria and characterize them as zeros of the relaxed defect; and
\item[vi)] solve exactly the relaxation of an simple-shear kinetic shear-banding model with local memory, identifying the history-space relaxed defect, proving pure recovery and existence, and relating two-rate coexistence to negative strain-rate sensitivity and dynamic strain aging.
\end{itemize}

Section~\ref{F2hxmX} derives the game and its diagonal reduction. Section~\ref{ZLpWEK} develops the abstract direct method and the diagonal-defect criteria. Section~\ref{sec:examples-attainment} verifies the theory in superlinear- and linear-growth settings, including the pathwise and measure-valued approximation constructions. Section~\ref{sec:direct-almost-equilibrium-relaxation} develops the enriched-space relaxation as the sequential closure of vanishing-defect pure paths. Section~\ref{sec:local-memory-banding} applies the construction to simple-shear kinetic shear banding with local memory and computes the sequential relaxation of the complete diagonal defect exactly. The concluding section summarizes the main results, connections with other work, computational implications, and possible extensions to more general causal evolution problems.

\section{Formal derivation of the evolution game} \label{F2hxmX}

This section makes the transition from the familiar instantaneous rate problem to the complete reduced payoff used throughout the paper. The game terminology makes explicit that the optimal rate at one time depends on the accumulated actions at earlier times. In the Volterra setting treated here, the unique state response can be eliminated exactly, so the analytical object after reduction is a diagonal variational relation on a path space rather than a continuum-player game in an abstract economic sense. This collapse is specific to the explicit reconstruction used below; for more general causal constraints it may be preferable or necessary to retain the coupled state--rate game. The formulation allows weak convergence and relaxation to act on the entire evolution at once. Throughout this section the derivation is formal; the spaces and hypotheses used in the rigorous results are specified in subsequent sections as needed.

\subsection{The classical rate problem}

We first recall the local rate formulation and isolate the causal consistency relation that will later be eliminated. Let $U$ be an affine space with translation vector space $V$. By an evolution we mean a map $u:[0,T]\to U$. We consider rate problems characterized by timewise minimization together with consistency:
\begin{subequations}\label{YrBoHs}
\begin{align}
    &   \label{BQpThW}
    v(t) \in \operatorname{argmin} \, \Phi(\cdot; u(t), t ) , 
    \quad t \in (0,T) ,
    \\ &    \label{uf5RHk}
    u(t) = u_0 + \int_0^t v(s) \, ds ,
    \quad t \in [0,T] .
\end{align}
\end{subequations}
Here $u_0\in U$ is the initial condition and $\Phi : V\times U\times(0,T) \to (-\infty,+\infty]$ is an extended-real-valued stage cost. At this formal stage we assume that all operations are meaningful.

A common energy--dissipation form is
\begin{equation} \label{xBmzun}
    \Phi(v;u,t) = \Psi(v;u) - \langle {f}(u,t) , v \rangle ,
\end{equation}
where $\Psi:V\times U\to(-\infty,+\infty]$ is a possibly state-dependent kinetic potential and ${f}:U\times(0,T)\to V^*$ is the driving force. If the force is conservative, there is an energy $E:U\times(0,T)\to\mathbb R$, differentiable in the state variable, such that ${f}(u,t)=-D_1E(u,t)$. Here and throughout, $D_1$ denotes differentiation with respect to the state variable. Then, the energy-dissipation functional
\begin{equation} \label{SB741X}
    \Phi(v;u,t) = \Psi(v;u) + \langle D_1E(u,t) , v \rangle
\end{equation}
encodes the instantaneous competition between dissipation and energy.

For convex $\Psi(\cdot;u)$, the rate condition may be written as a subdifferential inclusion and treated by monotone-operator or variational-inequality methods \cite{Minty1962, LerayLions1965, Brezis1968}. Gradient-flow and minimizing-movement formulations provide complementary constructions \cite{DeGiorgiMarinoTosques1980, AmbrosioGigliSavare2008}, while rate-independent systems replace superlinear dissipation by a one-homogeneous potential and an energy--dissipation balance \cite{MielkeTheil2004, MielkeRossiSavare2012}. The difficulty addressed here begins when the kinetic potential is non-convex in the rate and state/branch-dependent: pointwise minimizers may fail to exist and the missing compactness may involve both the rate and the state generated by it.

\subsection{Game-theoretical interpretation}

The family of minimizations in \eqref{YrBoHs} is indexed by time, but the choices are not independent. A rate selected at time $s$ changes the state seen by every later time $t>s$. This suggests the following heuristic game:
\begin{itemize}
\item[i)] the players are the times $t\in(0,T)$;
\item[ii)] the action of player $t$ is an instantaneous rate $v(t)\in V$;
\item[iii)] the state profile is the accumulated path $u(t)=u_0+\int_0^t v(s)\,ds$; and
\item[iv)] the stage cost of player $t$ is $\Phi(v(t);u(t),t)$.
\end{itemize}
The dependence on preceding actions has a leader--follower flavor reminiscent of dynamic Stackelberg games \cite{Stackelberg1934,BasarOlsder1999}. No continuum-player existence theorem is invoked here; the interpretation serves to expose causality and the best-response structure.

Formally, a pure equilibrium is a path satisfying
\begin{equation} \label{nOI8Kv}
    \Phi(\dot{u}(t);u(t),t) \leq \Phi(v;u(t),t) ,
    \quad
    \forall v \in V,
    \quad
    t \in (0,T) .
\end{equation}
The consistency constraint couples all players and is part of the Euler--Lagrange structure. Since the quantities in \eqref{nOI8Kv} are generally defined only almost everywhere, we proceed to replace the pointwise comparison (\ref{nOI8Kv}) by an integral one.

\subsection{Integral form of the game} \label{v5hrYD}

We now aggregate the timewise comparisons (\ref{nOI8Kv}) without losing their local content. Let ${X}$ be a state-path affine space and ${Y}$ a rate-path vector space, and restrict temporarily to paths $u\in {X}$ whose derivative belongs to ${Y}$. For every state $\xi\in U$ and time $t$, suppose, for the sake of the argument, that
\begin{equation} \label{oWSxJK}
    -\infty < m(\xi,t) = \inf_{\eta\in V} \Phi(\eta;\xi,t) < +\infty .
\end{equation}
For finite-cost pairs, further assume measurability and integrability of $\Phi(v(t);u(t),t)-m(u(t),t)\geq0$. For a measurable set $A\subset(0,T)$, define
\begin{equation} \label{EeOvri}
    G(v;u,A) := \int_A \Big( \Phi(v(t); u(t),t) - m(u(t),t) \Big) \, dt .
\end{equation}
The normalization by $m$ makes $G(v;u,\cdot)$ a nonnegative absolutely continuous measure. However, since $m$ is independent of the competitor rate, the normalization is optional and does not affect any best-response or diagonal comparison. In particular, subsequent applications allow competitor sections that, unlike (\ref{oWSxJK}), are unbounded below. 

The integral equilibrium condition is the measure domination
\begin{equation} \label{Nl3dw5}
    G(\dot{u};u,A) \leq G(v;u,A) ,
    \quad
    \forall v \in {Y},
    \quad
    A \in \mathcal{A}(0,T) ,
\end{equation}
where $\mathcal{A}(0,T)$ is the class of open subsets of $(0,T)$. Assume that ${Y}$ is stable under pasting: for every $v,z \in Y$ and $A\in\mathcal{A}(0,T)$,
\begin{equation} \label{ZKNHxn}
    w_A(t):=
    \begin{cases}
        v(t),& t\in A , \\
        z(t),& t\notin A
    \end{cases}
\end{equation}
belongs to ${Y}$. If the whole-interval comparison holds for every competitor, we can compare $\dot u$ with $w_A$ and cancel the identical contribution on $(0,T)\setminus A$ to obtain the comparison on $A$. The converse follows by choosing $A=(0,T)$. Consequently,
\begin{equation}
\text{(\ref{Nl3dw5})}
\;\,\Longleftrightarrow\;\,
G(\dot u;u,(0,T))\leq G(v;u,(0,T)) ,
\;\, \forall v\in {Y}.
\end{equation}
This measure comparison retains local information Lebesgue almost-everywhere. If, in addition, ${Y}$ contains the admissible measurable selections of pointwise competitors, it is equivalent to the almost-everywhere pointwise condition \eqref{nOI8Kv}. 

For the remainder of the paper we abbreviate $G(v;u,(0,T))$ as 
\begin{equation} \label{qdS8S8}
    G(v;u) := \int_0^T \Big( \Phi(v(t); u(t),t) - m(u(t),t) \Big) \, dt .
\end{equation}

At the path level, consistency can be represented as the indicator payoff
\begin{equation}
F(u;v)=I_\Gamma(u;v),
\quad
\Gamma=\{(Rv,v) \,:\, v\in {Y}\}\subset {X}\times {Y}
\end{equation}
where the reconstruction operator $R : {Y} \to {X}$ and the Volterra operator $K : {Y} \to {X}$ are defined as
\begin{equation} \label{7TzojC}
    Rv(t) := u_0+Kv(t) , 
    \quad
    Kv(t):=\int_0^t v(s)\,ds ,
\end{equation}
respectively. 

For economy of notation, operator application is understood before time evaluation: $Rv(t):=(Rv)(t)$ and $Kv(t):=(Kv)(t)$, and similarly for the other operators used below. 

The collective Nash conditions are therefore
\begin{subequations} \label{eq:Nash}
\begin{align}
    &   \label{eq:Nash-F}
    F(u^*;v^*)\leq F(u;v^*) , &&\forall u\in {X},
    \\ &    \label{eq:Nash-G}
    G(v^*;u^*)\leq G(v;u^*) , &&\forall v\in {Y}.
\end{align}
\end{subequations}
Fig.~\ref{fig:game-problem-3-2} illustrates this two-player best-response structure in a simple finite-dimensional setting. However, in the present setting the first condition (\ref{eq:Nash-F}) has a unique finite response, namely $u^*=Rv^*$. Hence, the state player can be eliminated exactly rather than handled as a part of a generic two-player best-response system. 

\begin{figure}[h]
  \centering
  \includegraphics[width=0.5\textwidth]{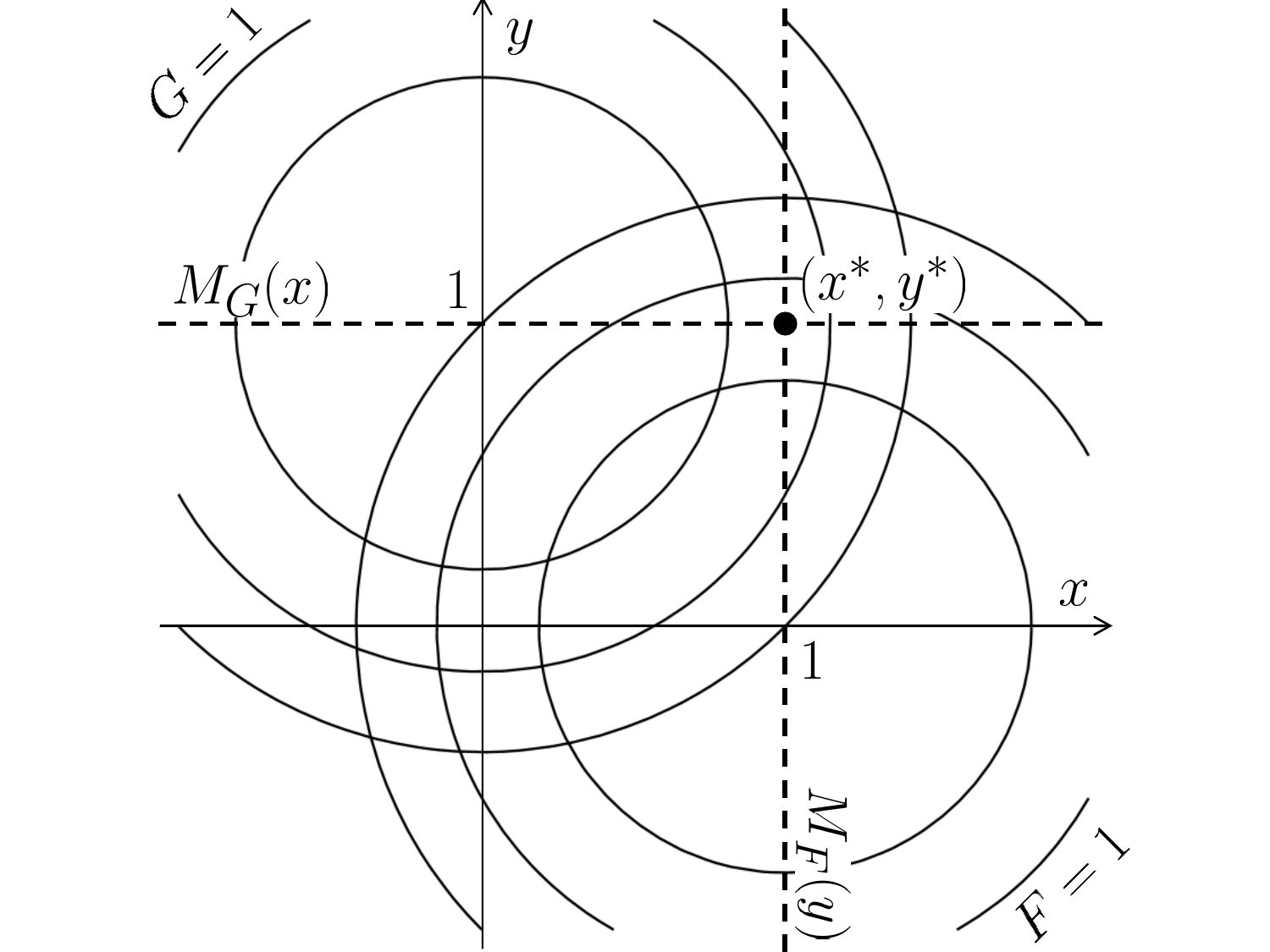}
  \caption{Two-player best-response picture before the consistency reduction, for
  ${X}={Y}=\mathbb{R}$,
  $F(x;y)=\frac12(x-1)^2+\frac12y^2$, and
  $G(y;x)=\frac12x^2+\frac12(y-1)^2$. The best responses intersect at
  $(x^*,y^*)$. In the evolution problem the state response is the
  graph $u=Rv$; eliminating it produces the diagonal reduced relation
  studied below.}
  \label{fig:game-problem-3-2}
\end{figure}

\subsection{Diagonal reduction} \label{e56Lxq}

For $v,w\in {Y}$, set
\begin{equation}
\label{eq:complete-reduced-payoff-formal}
    {J}(w;v):=G(w;Rv).
\end{equation}
The first argument is the competitor rate; the second, written after the semicolon, is the rate that generates the state. Conditions \eqref{eq:Nash} are equivalent to the single diagonal comparison principle
\begin{equation}
\label{eq:formal-diagonal-equilibrium}
    {J}(v^*;v^*)\leq {J}(w;v^*),
    \quad\forall w\in {Y}.
\end{equation}
If $w\mapsto{J}(w;v)$ is proper, convex, and lower-semicontinuous, then \eqref{eq:formal-diagonal-equilibrium} is equivalent to the inclusion problem
\begin{equation} \label{6ZtPzU}
    0\in L(v^*) , \quad L(v):=\partial_1{J}(v;v)
\end{equation}
where $L(v)$ is the complete diagonal response and, for $v,w\in {Y}$, 
\begin{equation}
    \partial_1J(w;v)
    :=
    \{
    {f}\in {Y}^*:
    J({z};v)
    \geq
    J(w;v)
    +
    \langle {f},{z}-w\rangle_{{Y}^*,{Y}}, \;\,
    \forall {z}\in {Y}
    \}
\end{equation}
is the convex subdifferential of the section $J(\cdot;v)$ at $w$. Crucially, both approaches test dissipation and energy together. 

\section{A game-theoretical direct method}
\label{ZLpWEK}

The diagonal reduction of Section~\ref{F2hxmX} turns the equilibrium problem into the problem of making a nonnegative diagonal defect vanish. The direct method can then be applied provided that approximate diagonal equilibria exist, the corresponding defect is coercive, and the two occurrences of the rate variable have the complementary semicontinuity properties required to pass to the limit. This compactness--semicontinuity architecture is the familiar one behind direct variational and monotonicity methods, adapted here to a causal diagonal relation; compare \cite{BrowderHess1972,LerayLions1965,Brezis1968,Zeidler1990,Showalter1997}. The approximability condition is essential: coercivity and semicontinuity alone do not force the diagonal defect to have vanishing infimum.

The first result keeps the two arguments of $J$ separate. An almost-equilibrium sequence supplies the approximants, compactness produces a candidate limit, and the lower/upper semicontinuity assumptions pass the diagonal and competitor sides in opposite directions. This argument is analogous in spirit to standard equilibrium and variational-inequality existence arguments, see, e.g., \cite{Fan1961, BlumOettli1994}.

\begin{proposition}[Direct existence of a diagonal equilibrium] \label{aBgxdV}
Let $({Y},d)$ be a metric space, and let $J:{Y}\times {Y}\to(-\infty,+\infty]$ satisfy $J(v;v)<+\infty$ for every $v\in {Y}$. Suppose, in addition, that the following conditions hold:
\begin{enumerate}
\item[\rm(i)] \emph{Approximate diagonal equilibria:}
For every $\varepsilon>0$, there exists $v_{\varepsilon}\in {Y}$ such that
\begin{equation}
J(v_{\varepsilon};v_{\varepsilon})
\leq J(w;v_{\varepsilon})+\varepsilon ,
\quad\forall w\in {Y}.
\end{equation}
\item[\rm(ii)] \emph{Sequential compactness:}
There exists $a>0$ such that the set
\begin{equation}
K_a:=
\left\{
v\in {Y}:
J(v;v)\leq J(w;v)+a , 
\quad \forall w\in {Y}
\right\}
\end{equation}
is sequentially compact in ${Y}$.
\item[\rm(iii)] \emph{Diagonal sequential lower semicontinuity:}
For every sequence $(v_h)\subset K_a$ such that $v_h\to v$ in ${Y}$,
\begin{equation}
J(v;v)
\leq
\liminf_{h\to\infty}J(v_h;v_h).
\end{equation}

\item[\rm(iv)] \emph{Sequential upper semicontinuity in the second argument:}
For every $w\in {Y}$ and every sequence $(v_h)\subset K_a$ such that
$v_h\to v$ in ${Y}$,
\begin{equation}
\limsup_{h\to\infty}J(w;v_h)
\leq
J(w;v).
\end{equation}
\end{enumerate}

Then, there exists $v^*\in {Y}$ such that
\begin{equation}
J(v^*;v^*)
\leq
J(w;v^*) , 
\quad \forall w\in {Y}.
\end{equation}
Thus, $v^*$ is a diagonal equilibrium.
\end{proposition}

\begin{proof}
Let $\varepsilon_h\downarrow0$. By assumption {\rm(i)}, for every $h$
there exists $v_h\in {Y}$ such that
\begin{equation}
J(v_h;v_h)
\leq
J(w;v_h)+\varepsilon_h , 
\quad \forall w\in {Y}.
\end{equation}
Since $\varepsilon_h\leq a$ for all sufficiently large $h$, after discarding finitely many terms we have $v_h\in K_a$. Assumption {\rm(ii)} therefore yields a subsequence, not relabeled, and a point $v^*\in K_a$ such that $v_h\to v^*$ in $Y$. Fix $w\in {Y}$. By assumptions {\rm(iii)} and {\rm(iv)},
\begin{equation}
J(v^*;v^*)
\leq \liminf_{h\to\infty}J(v_h;v_h) 
\leq \limsup_{h\to\infty}
\bigl(J(w;v_h)+\varepsilon_h\bigr) 
=\limsup_{h\to\infty}J(w;v_h) 
\leq J(w;v^*).
\end{equation}
Since $w\in {Y}$ is arbitrary, $v^*$ is a diagonal equilibrium.
\end{proof}

An alternative characterization of existence can be based on the \emph{diagonal defect function}. Define the best-response value
\begin{equation}
\label{eq:best-response-value}
    m(v):=\inf_{w\in {Y}} J(w;v),
    \quad
    v\in {Y},
\end{equation}
and the diagonal defect
\begin{equation}
\label{eq:diagonal-defect}
    I(v):=
    \begin{cases}
    J(v;v)-m(v),&m(v)>-\infty,\\
    +\infty,&m(v)=-\infty.
    \end{cases}
\end{equation}
Alternatively, 
\begin{equation}
I(v):=\sup_{w\in Y}\bigl\{J(v;v)-J(w;v)\bigr\}.
\end{equation}
Since $J(v;v)<+\infty$, one has $m(v)\in[-\infty,\mathbb R]$ and $I(v)\in[0,+\infty]$. Moreover,
\begin{equation}
\label{eq:defect-equilibrium-equivalence}
    I(v)=0
    \quad\Longleftrightarrow\quad
    J(v;v)\leq J(w;v) ,
    \quad\forall w\in {Y}.
\end{equation}

Thus, zeros of $I$ are precisely the diagonal equilibria of \eqref{eq:formal-diagonal-equilibrium}. We denote by
\begin{equation}
\label{eq:equilibrium-set}
M(J)
:=
\bigl\{
v\in {Y} \,:\, I(v)=0
\bigr\}
\end{equation}
the set of diagonal equilibria. No attainment of $m(v)$ is required. The extended-value convention also covers competitor sections that are unbounded below.

The scalar defect packages the best-response comparison into a single nonnegative functional. Once this reduction has been made, the existence argument takes the classical direct-method form: construct a minimizing sequence with value tending to zero, extract a convergent subsequence, and use lower semicontinuity to retain the zero value in the limit.

\begin{proposition}[Direct existence from the diagonal defect]\label{prop:direct-defect}
Let $({Y},d)$ be a metric space, and let $I:{Y}\to[0,+\infty]$. Suppose that the following conditions hold:
\begin{enumerate}
\item[\rm(i)] \emph{Approximability:}
\begin{equation}
\inf_{v\in {Y}} I(v)=0.
\end{equation}
\item[\rm(ii)] \emph{Sequential compactness:}
There exists $a>0$ such that the sublevel set
\begin{equation}
K_a:=\{v\in {Y}:I(v)\leq a\}
\end{equation}
is sequentially compact in ${Y}$.
\item[\rm(iii)] \emph{Sequential lower semicontinuity:}
For every sequence $(v_h)\subset K_a$ such that $v_h\to v$ in ${Y}$,
\begin{equation}
I(v)\leq\liminf_{h\to\infty}I(v_h).
\end{equation}
\end{enumerate}
Then, there exists $v^*\in K_a$ such that
\begin{equation}
I(v^*)=0.
\end{equation}
In particular, $I$ attains its minimum on ${Y}$.
\end{proposition}

\begin{proof}
By assumption {\rm(i)}, choose a sequence $(v_h)\subset {Y}$ such that $I(v_h)\to0$. Since $a>0$, after discarding finitely many terms we have $v_h\in K_a$. Assumption {\rm(ii)} therefore yields a subsequence, not relabeled, and $v^*\in K_a$ such that $v_h\to v^*$ in ${Y}$. Assumption {\rm(iii)} gives $0\leq I(v^*)\leq\liminf_{h\to\infty}I(v_h)=0$, hence $I(v^*)=0$.
\end{proof}

The two formulations are not independent. The next corollary shows that the complementary semicontinuity assumptions on $J$ automatically supply lower semicontinuity of the induced defect $I$, so Proposition~\ref{prop:direct-defect} recovers the preceding $J$-based criterion.

\begin{corollary}[The $J$-criterion implies the $I$-criterion] \label{cor:J-implies-I}
Let $(Y,d)$ and $J:Y\times Y\to(-\infty,+\infty]$ satisfy the assumptions of Proposition~\ref{aBgxdV}. Then, $I:Y\to[0,+\infty]$ and, for every $\varepsilon\geq 0$,
\begin{equation}
I(v)\leq\varepsilon
\quad\Longleftrightarrow\quad
J(v;v)\leq J(w;v)+\varepsilon ,
\quad\forall w\in Y.
\end{equation}
Moreover, $I$ satisfies the assumptions of Proposition~\ref{prop:direct-defect}. Consequently, there exists $v^*\in Y$ such that $I(v^*)=0$, or, equivalently,
\begin{equation}
J(v^*;v^*)\leq J(w;v^*) ,
\quad\forall w\in Y.
\end{equation}
\end{corollary}

\begin{proof}
Since $w=v$ is admissible in the definition of $I$, one has $I(v)\geq J(v;v)-J(v;v)=0$, and hence $I:Y\to[0,+\infty]$. Moreover, by the definition of the supremum, $I(v)\leq\varepsilon$ if and only if $J(v;v)-J(w;v)\leq\varepsilon$ for every $w\in Y$, which proves the stated equivalence.

The approximability assumption of Proposition~\ref{aBgxdV} therefore gives, for every $\varepsilon>0$, some $v_\varepsilon\in Y$ with $I(v_\varepsilon)\leq\varepsilon$, and hence $\inf_{v\in Y}I(v)=0$. Likewise, for the $a>0$ in Proposition~\ref{aBgxdV}, the sublevel set $\{v\in Y:I(v)\leq a\}$ coincides with the corresponding almost-equilibrium set $K_a$ and is therefore sequentially compact.

It remains to prove sequential lower semicontinuity on $K_a$. Let $(v_h)\subset K_a$ with $v_h\to v$ in $Y$. For every fixed $w\in Y$, $I(v_h)\geq J(v_h;v_h)-J(w;v_h)$, and the semicontinuity assumptions of Proposition~\ref{aBgxdV} give
\begin{equation}
\liminf_{h\to\infty}I(v_h)
\geq \liminf_{h\to\infty}J(v_h;v_h)-\limsup_{h\to\infty}J(w;v_h)
\geq J(v;v)-J(w;v).
\end{equation}
Taking the supremum over $w\in Y$ yields $\liminf_{h\to\infty}I(v_h)\geq I(v)$. Thus $I$ satisfies all the assumptions of Proposition~\ref{prop:direct-defect}, which gives $v^*\in Y$ such that $I(v^*)=0$. By the stated equivalence with $\varepsilon=0$, this is equivalent to $J(v^*;v^*)\leq J(w;v^*)$ for every $w\in Y$.
\end{proof}

\begin{remark}[Reflexive Banach spaces]
{\rm The preceding argument applies equally with weak convergence on a reflexive Banach space $Y$. If $\inf_{v\in Y} I(v)=0$, $I$ is coercive, and $I$ is weakly sequentially lower semicontinuous on the relevant defect sublevel, then coercivity and reflexivity give the weak sequential compactness required in Proposition~\ref{prop:direct-defect}, and hence a diagonal equilibrium exists. By Corollary~\ref{cor:J-implies-I}, the corresponding weak sequential assumptions on $J$ give the same conclusion. No metrizability of the weak topology is required.}\hfill$\square$
\end{remark}

\begin{remark}[Necessity of approximability]
{\rm The condition $\inf_{v\in {Y}} I(v)=0$ is not a consequence of coercivity and semicontinuity. For example, let ${Y}=\mathbb{R}$ and define
\begin{equation}
    f(v):=v-\sqrt{v^2+1},
    \quad
    J(w;v):=(w-f(v))^2.
\end{equation}
Then, \begin{equation}
    m(v)=0,
    \quad
    I(v)=J(v;v)=v^2+1.
\end{equation}
The defect $I$ is continuous and coercive, and $J$ is continuous in both variables, but $\inf_{v\in {Y}} I(v)=1$. A diagonal equilibrium would require $v=f(v)$, which is impossible. Hence the zero-defect approximability assumption is indispensable for the direct argument above.}\hfill$\square$
\end{remark}

\section{Examples of attainment}
\label{sec:examples-attainment}

The abstract criteria of Section~\ref{ZLpWEK} separate three issues: the existence of vanishing-defect competitors, compactness of the resulting almost equilibria, and semicontinuity of the two occurrences of the rate. The next two examples show how these ingredients arise in two different settings. Section~\ref{subsec:superlinear-energy-dissipation} uses weak compactness in a reflexive Bochner space. Section~\ref{subsec:linear-growth-energy-dissipation} is formulated from the outset on $H$-valued Radon measures and right-continuous $BV$ states. A centered state representative gives the exact quadratic chain rule at atoms, while the energy identity supplies a data-dependent total-variation bound and hence weak-star compactness. The $W^{1,1}$ regularity of the loading is used only afterward to prove that the limiting equilibrium measure is absolutely continuous.

\subsection{Superlinear-growth energy--dissipation systems}
\label{subsec:superlinear-energy-dissipation}

We begin with an example that illustrates existence of equilibria in a reflexive Banach space setting. 

Let $H$ be a real Hilbert space, let $\alpha>0$, and let $1<p\leq2$, with conjugate exponent $p'=p/(p-1)$. Set ${X}:=W^{1,p}(0,T;H)$ and ${Y}:=L^p(0,T;H)$.
Fix $u_0\in H$. The reconstruction operator $R$ in \eqref{7TzojC} maps ${Y}$ continuously into ${X}$, with $Rv(0)=u_0$ and $\dot{Rv}=v$ almost everywhere. Let $S\in\mathcal L(H)$ be self-adjoint and positive semidefinite, and let $h\in L^{p'}(0,T;H)$. Consider the dissipation and energy functions
\begin{equation}
\label{eq:weighted-quadratic-data}
\Psi(\eta):=\alpha\|\eta\|_H^p,
\quad
E(\xi,t):=\frac12(S\xi,\xi)_H-(h(t),\xi)_H ,
\end{equation}
respectively. The corresponding complete payoff is
\begin{equation}
\label{eq:superlinear-growth-complete-payoff}
J(w;v)
:=
\int_0^T
\Big(
\alpha\|w(t)\|_H^p
+
\bigl(S Rv(t)-h(t),w(t)\bigr)_H
\Big)\,dt .
\end{equation}
The section $J(\cdot;v)$ is proper, coercive, weakly lower semicontinuous, and strictly convex on ${Y}$. These are the standard convex-analytic properties used in monotone-operator and convex integral-functional theory \cite{BrowderHess1972,Zeidler1990,Rockafellar1971IntegralII,Rockafellar1971ConvexIntegralDuality}. The point specific to the present formulation is that the Volterra contribution remains inside the diagonal defect and supplies the nonnegative term used below.

The next proposition verifies the abstract architecture and, at the same time, produces a constructive pathwise approximation.

\begin{proposition}[Direct existence, uniqueness, and temporal Galerkin approximation]
\label{prop:superlinear-attainment}
Under the preceding assumptions, there exists a unique pair $(u^*,v^*)\in {X}\times {Y}$ such that $u^*(0)=u_0$, $u^*=Rv^*$, and
\begin{equation}
\label{eq:superlinear-diagonal-equilibrium}
J(v^*;v^*)
\leq
J(w;v^*) ,
\quad
\forall w\in {Y}.
\end{equation}
Equivalently, $u^*$ is the unique state path in ${X}$ satisfying $u^*(0)=u_0$ and
\begin{equation}
\label{eq:superlinear-state-equation}
\alpha p\|\dot u^*(t)\|_H^{p-2}\dot u^*(t)
+
Su^*(t)-h(t)
=
0 ,
\quad
\text{for a.e. } t\in(0,T),
\end{equation}
where the first term is understood to be zero when $\dot u^*(t)=0$. Moreover, there is a sequence $v_n\to v^*$ strongly in ${Y}$, such that $I(v_n)\to I(v^*)=0$.
\end{proposition}

\begin{proof}
We begin by examining the diagonal defect and its basic properties. Define
\begin{equation}
\Delta(v):=\alpha\|v\|_{L^p}^p,
\quad
\Delta^*(q):={\alpha'}\|q\|_{L^{p'}}^{p'},
\quad
{\alpha'}:=\frac{1}{p'}(\alpha p)^{-1/(p-1)},
\end{equation}
and set $g_v:=S Rv-h$ and $A:=SK$. For every $z\in {Y}$, self-adjointness and positive semidefiniteness of $S$ give
\begin{equation}
\label{eq:superlinear-volterra-positive}
\langle Az,z\rangle
=
\int_0^T(SKz(t),z(t))_H\,dt
=
\frac12(SKz(T),Kz(T))_H
\geq0.
\end{equation}
Consequently,
\begin{equation}
\label{eq:superlinear-diagonal-representation}
J(v;v)
=
\alpha\|v\|_{L^p}^p
+
\frac12(SKv(T),Kv(T))_H
+
\int_0^T(Su_0-h(t),v(t))_H\,dt.
\end{equation}
Convex duality for integral functionals \cite{Rockafellar1971IntegralII, Rockafellar1971ConvexIntegralDuality} gives
\begin{equation}
\inf_{w\in {Y}}J(w;v)
=
-\Delta^*(-g_v)
=
-{\alpha'}\|g_v\|_{L^{p'}}^{p'}.
\end{equation}
Therefore,
\begin{equation}
\label{eq:superlinear-defect-representation}
\begin{aligned}
I(v)
=&
\alpha\|v\|_{L^p}^p
+
\frac12(SKv(T),Kv(T))_H
+
\int_0^T(Su_0-h(t),v(t))_H\,dt
+ 
{\alpha'}\|S Rv-h\|_{L^{p'}}^{p'}.
\end{aligned}
\end{equation}
Because $1<p\leq2$, the Volterra operator is bounded from $L^p(0,T;H)$ to $L^{p'}(0,T;H)$. It follows from \eqref{eq:superlinear-defect-representation} that $I$ is finite and strongly continuous on ${Y}$. The same representation shows that $I$ is strictly convex and weakly sequentially lower semicontinuous. Moreover,
\begin{equation} \label{AYEfAd}
I(v)
\geq
J(v;v)
\geq
\alpha\|v\|_{L^p}^p
-
\|Su_0-h\|_{L^{p'}}\|v\|_{L^p},
\end{equation}
and $I$ is coercive.

Next, we introduce temporal Galerkin rate-path subspaces and corresponding minimizers. For $n\in\mathbb N$, let
\begin{equation}
0=t_0^n<t_1^n<\cdots<t_{2^n}^n=T,
\quad
\tau_n:=\frac{T}{2^n},
\quad
I_i^n:=(t_{i-1}^n,t_i^n),
\quad
{Y}_n
:=
\left\{
\sum_{i=1}^{2^n}\eta_i\chi_{I_i^n}:
\eta_i\in H
\right\}.
\end{equation}
The spaces ${Y}_n$ are nested closed subspaces of ${Y}$, and their union is strongly dense in ${Y}$. The approximation is Galerkin only in the time variable: the coefficients $\eta_i$ remain arbitrary elements of $H$, so ${Y}_n$ need not be finite-dimensional when $H$ is infinite-dimensional. We use the term \emph{temporal Galerkin} in this sense. Let $P_n$ denote interval averaging,
\begin{equation}
(P_nf)(t)
:=
\frac{1}{\tau_n}\int_{I_i^n}f(s)\,ds
\quad
\text{for }t\in I_i^n.
\end{equation}
Then, $P_n$ is a contraction on $L^q(0,T;H)$ for every $1\leq q<\infty$, and $P_nf\to f$ strongly in $L^q(0,T;H)$.

Since ${Y}_n$ is weakly closed, the direct method applied to $I|_{{Y}_n}$ gives a unique minimizer $v_n$. We proceed to show that the sequence $(v_n)$ has a weak limit. Because $0\in {Y}_n$, $I(v_n)\leq I(0)$, and coercivity shows that $(v_n)$ is bounded in ${Y}$. By reflexivity, along a subsequence, not relabeled, $v_n\rightharpoonup \bar v$ in ${Y}$. Fix $w\in {Y}$. Since $P_nw\in {Y}_n$, minimality, weak lower semicontinuity of $I$, and strong continuity of $I$ give
\begin{equation}
I(\bar v)
\leq
\liminf_{n\to\infty}I(v_n)
\leq
\limsup_{n\to\infty}I(v_n)
\leq
\lim_{n\to\infty}I(P_nw)
=
I(w).
\end{equation}
Thus, $\bar v$ is the unique global minimizer of $I$.

It remains to show that the minimum value is zero. For $v\in {Y}$, let $b(v):=D\Delta^*(-g_v)$. Then, $b(v)$ is the unique minimizer of $J(\cdot;v)$ and satisfies $D\Delta(b(v))+g_v=0$. The defect is Fr\'echet differentiable, with
\begin{equation}
DI(v)[z]
=
\langle D\Delta(v)+g_v,z\rangle
+
\langle Az,v-b(v)\rangle ,
\quad
\forall z\in {Y}.
\end{equation}
Set $z:=\bar v-b(\bar v)$. Since $DI(\bar v)=0$,
\begin{equation}
\begin{aligned}
0
=&
\langle D\Delta(\bar v)-D\Delta(b(\bar v)),z\rangle
+
\langle Az,z\rangle.
\end{aligned}
\end{equation}
The first term is nonnegative by strict monotonicity of $D\Delta$ and vanishes only when $z=0$; the second is nonnegative by \eqref{eq:superlinear-volterra-positive}. Hence $\bar v=b(\bar v)$, so $D\Delta(\bar v)+g_{\bar v}=0$ and $I(\bar v)=0$. Set $v^*:=\bar v$. Since $J(\cdot;v^*)$ is strictly convex, the first equality is equivalent to \eqref{eq:superlinear-diagonal-equilibrium}, and it gives the equilibrium equation
\begin{equation}
\label{eq:superlinear-operator}
\alpha p\|v^*(t)\|_H^{p-2}v^*(t)
+
S Rv^*(t)-h(t)
=
0 ,
\quad
\text{for a.e. }t\in(0,T).
\end{equation}
Every diagonal equilibrium is a zero of $I$, and the strict convexity of $I$ makes $v^*$ its unique zero, which establishes the uniqueness of the equilibrium. Setting $u^*:=Rv^*$ gives \eqref{eq:superlinear-state-equation} and uniqueness of the pair.

Finally, we investigate further the convergence properties of the temporal Galerkin approximations. The preceding compactness argument applies to every subsequence of $(v_n)$. Since every weak cluster point is the unique minimizer $v^*$, it follows that $v_n\rightharpoonup v^*$ in ${Y}$. Furthermore, $P_nv^*\in {Y}_n$, and
\begin{equation}
0
\leq
I(v_n)
\leq
I(P_nv^*)
\to
I(v^*)
=
0.
\end{equation}
This proves $I(v_n)\to I(v^*)=0$. The linear term in \eqref{eq:superlinear-defect-representation} converges under weak convergence. Hence,
\begin{equation}
\begin{aligned}
&\alpha\|v_n\|_{L^p}^p
+
\frac12(SKv_n(T),Kv_n(T))_H
+
{\alpha'}\|S Rv_n-h\|_{L^{p'}}^{p'}
\\
&\quad\to
\alpha\|v^*\|_{L^p}^p
+
\frac12(SKv^*(T),Kv^*(T))_H
+
{\alpha'}\|S Rv^*-h\|_{L^{p'}}^{p'}.
\end{aligned}
\end{equation}
Each term is weakly lower semicontinuous. A standard liminf argument therefore yields $\|v_n\|_{L^p} \to \|v^*\|_{L^p}$. Since ${Y}=L^p(0,T;H)$ is uniformly convex, $v_n\to v^*$ strongly in ${Y}$.
\end{proof}

\begin{remark}[Relation with the abstract criteria]
{\rm The proof verifies the hypotheses of Propositions~\ref{aBgxdV} and \ref{prop:direct-defect}. Indeed, the temporal Galerkin sequence satisfies $I(v_n)\to0$; coercivity and reflexivity give weak sequential compactness of the relevant defect and almost-equilibrium sets; and (\ref{eq:superlinear-diagonal-representation}--\ref{eq:superlinear-defect-representation}) provide the required weak lower semicontinuity of the diagonal terms and weak continuity of $v\mapsto J(w;v)$ for fixed $w$. Thus, Proposition~\ref{prop:superlinear-attainment} provides a concrete realization of both abstract direct methods.}\hfill$\square$
\end{remark}

\begin{remark}[Alternative constructions of approximate equilibria]
{\rm The temporal Galerkin minimizers $v_n$ above minimize the diagonal defect over $Y_n$, but are not, in general, exact equilibria of the restricted game: differentiating $I|_{Y_n}$ produces an additional adjoint term from the state dependence. Exact restricted equilibria can instead be obtained by an interval-wise causal construction. Other possibilities include Picard iteration or vanishing temporal or monotone regularization, followed by compactness and strong convergence to obtain vanishing unregularized defect. The pathwise defect minimization used here is particularly convenient because it combines directly with the strong density of $Y_n$ and the weak lower semicontinuity of $I$.}\hfill$\square$
\end{remark}

\subsection{Linear-growth energy--dissipation systems in measure form}
\label{subsec:linear-growth-energy-dissipation}

We continue with an example that illustrates existence of equilibria in a metric space setting. The metric structure follows from the metrizability of weak-star convergence of measures in bounded sets, the bound being supplied by energy balance. 

At linear growth, the rate path is naturally a measure and the state path is naturally a function of bounded variation. We therefore take these as the ambient spaces from the outset. Absolutely continuous rate paths enter only as a regularizing and temporal approximation class; they are not built into the solution concept. A central point is that the work term must use the centered representative of the reconstructed state, see, e.g., \cite{AmbrosioDalMaso1990, AFP00}. This convention accounts exactly for the quadratic energy carried by atoms of the rate measure and gives a well-defined diagonal payoff for arbitrary measure-valued evolutions.

Let $H$ be a separable real Hilbert space, fix $u_0\in H$, let $\alpha>0$, and set
\begin{equation}
\label{eq:linear-measure-spaces}
\mathcal X
:=
\left\{
u\in BV([0,T];H):
u\text{ right-continuous and } u(0^-)=u_0
\right\} ,
\quad
\mathcal Y
:=
\mathcal M([0,T];H).
\end{equation}
Here $\mathcal M([0,T];H)$ denotes the space of finite $H$-valued Radon measures of bounded variation. Since $H^*=H$ is reflexive and therefore has the Radon--Nikodym property, the vector-valued Riesz representation identifies $\mathcal Y$ isometrically with $C([0,T];H)^*$ \cite{DiestelUhl1977} through the pairing:
\begin{equation}
\langle \nu,\varphi\rangle
:=
\int_{[0,T]}(\varphi(t),d\nu(t))_H.
\end{equation}
Because $H$ is separable, $C([0,T];H)$ is separable. Banach--Alaoglu and metrizability of weak-star compact subsets of the dual of a separable Banach space therefore imply that every total-variation-bounded subset of $\mathcal Y$ is weak-star sequentially compact; see, e.g., \cite[Theorem~3.28]{Brezis2011}.

Throughout this subsection, $Du$ denotes the $H$-valued Lebesgue--Stieltjes measure associated with a right-continuous $u\in BV([0,T];H)$ and its prescribed left trace $u(0^-)=u_0$, characterized by
\begin{equation} \label{ESUERM}
Du([0,t])=u(t)-u(0^-),
\quad
Du((s,t])=u(t)-u(s),
\quad
0\leq s<t\leq T.
\end{equation}
This convention retains possible atoms at both endpoints. Correspondingly, $\operatorname{Var}_H(u;[0,T])$ is understood relative to the prescribed left trace, so that it includes the initial jump $\|u(0)-u(0^-)\|_H$. 

For $\nu\in\mathcal Y$, define the reconstruction operator as
\begin{equation}
\label{eq:linear-measure-reconstruction}
{R\nu}(t)
:=
u_0+\nu([0,t]),
\quad
{R\nu}(0^-):=u_0.
\end{equation}
Then, $R\nu\in\mathcal X$,
\begin{equation}
\label{eq:linear-measure-variation}
D{R\nu}=\nu,
\quad
\operatorname{Var}_H(R\nu;[0,T])
=
|\nu|([0,T]).
\end{equation}
We recall that, since $\nu \in \mathcal M([0,T];H)$ has finite total variation, it can have at most countably many atoms, see, e.g., \cite[Chapter~I]{DiestelUhl1977}. At an atom, $\nu(\{t\})$ is the jump ${R\nu}(t)-{R\nu}(t^-)$. We use the centered representative
\begin{equation}
\label{eq:linear-centered-reconstruction}
\widehat{R\nu}(t)
:=
\frac12\bigl({R\nu}(t^-)+{R\nu}(t)\bigr)
=
 u_0+\nu([0,t))+\frac12\nu(\{t\}) ,
\end{equation}
which defines a bounded Borel map on $[0,T]$.

Let $S\in\mathcal L(H)$ be self-adjoint and uniformly positive: there exists $\lambda>0$ such that
\begin{equation}
\label{eq:linear-measure-coercivity}
(S\xi,\xi)_H
\geq
\lambda\|\xi\|_H^2 ,
\quad
\forall\xi\in H.
\end{equation}
Let $h\in W^{1,1}(0,T;H)$ and use its continuous representative on $[0,T]$. The stored energy and the measure dissipation are then
\begin{equation}
\label{eq:linear-measure-energy-dissipation}
\mathcal E(\xi,t)
:=
\frac12(S\xi,\xi)_H-(h(t),\xi)_H,
\quad
\mathcal D(\mu;[0,t])
:=
\alpha|\mu|([0,t]).
\end{equation}
We impose the initial stability condition
\begin{equation}
\label{eq:linear-measure-initial-stability}
\|Su_0-h(0)\|_H
\leq
\alpha.
\end{equation}

For $\mu,\nu\in\mathcal Y$, set $g_\mu(t):=S\widehat{R\mu}(t)-h(t)$ and define the complete measure payoff
\begin{equation}
\label{eq:linear-measure-payoff}
\mathcal J(\nu;\mu)
:=
\alpha|\nu|([0,T])
+
\int_{[0,T]}(g_\mu(t),d\nu(t))_H ,
\end{equation}
see \cite{Reshetnyak1968}. Specifically, the last term means
\begin{equation}
\int_{[0,T]}(g_\mu,d\nu)_H
:=
\int_{[0,T]}(g_\mu(t),p_\nu(t))_H\,d|\nu|(t) ,
\end{equation}
where $\nu=p_\nu|\nu|$ is the polar decomposition of $\nu$. Thus, \eqref{eq:linear-measure-payoff} is finite for every pair of finite measures. For $A\in \mathcal{B}([0,T])$, the Borel $\sigma$-algebra on $[0,T]$, $\mathcal J_A$ denotes the same expression with both measure terms restricted to $A$.

If $u=R\mu$, then for every $t\in[0,T]$,
\begin{equation}
\label{eq:linear-quadratic-measure-chain-rule}
\int_{[0,t]}(S\widehat u(s),d\mu(s))_H
=
\frac12(Su(t),u(t))_H
-
\frac12(Su_0,u_0)_H ,
\end{equation}
which shows that the centered representative \eqref{eq:linear-centered-reconstruction} is not merely a convenient convention: it is uniquely selected among all representatives of the form $(1-\theta)u(t^-)+\theta u(t)$ by requiring the exact quadratic chain rule at every atom. This identity is the quadratic specialization of the general chain-rule calculus for $BV$ maps and distributional derivatives \cite{AmbrosioDalMaso1990,AFP00}. Integration by parts against the absolutely continuous loading further gives
\begin{equation}
\label{eq:linear-loading-measure-chain-rule}
\int_{[0,t]}(h(s),d\mu(s))_H
=
(h(t),u(t))_H-(h(0),u_0)_H
-
\int_0^t(\dot h(s),u(s))_H\,ds.
\end{equation}
Consequently, the localized diagonal payoff is exactly the energy residual
\begin{equation}
\label{eq:linear-diagonal-energy-representation}
\mathcal J_{[0,t]}(\mu;\mu)
=
\mathcal E(u(t),t)-\mathcal E(u_0,0)
+
\alpha|\mu|([0,t])
+
\int_0^t(\dot h(s),u(s))_H\,ds.
\end{equation}
In particular, atoms require no separate jump correction: their full quadratic contribution is already contained in \eqref{eq:linear-quadratic-measure-chain-rule}.

For subsequent use, we define the safe-stress violation and the bounded-competitor residual as
\begin{equation}
\label{eq:linear-overstress}
\mathcal Q(\mu)
:=
\sup_{t\in[0,T]}
\bigl(\|g_\mu(t)\|_H-\alpha\bigr)_+,
\end{equation}
and
\begin{equation}
\label{eq:linear-bounded-defect}
\mathcal I_M(\mu)
:=
\sup_{\substack{\nu\in\mathcal Y\\ |\nu|([0,T])\leq M}}
\bigl\{
\mathcal J(\mu;\mu)-\mathcal J(\nu;\mu)
\bigr\},
\quad
M>0 ,
\end{equation}
respectively. Point-mass competitors and the polar decomposition give
\begin{equation}
\label{eq:linear-bounded-defect-representation}
\inf_{\substack{\nu\in\mathcal Y\\ |\nu|([0,T])\leq M}}
\mathcal J(\nu;\mu)
=
-M\mathcal Q(\mu),
\quad
\mathcal I_M(\mu)
=
\mathcal J(\mu;\mu)+M\mathcal Q(\mu).
\end{equation}
Thus, $\mathcal Q$ measures failure of stability, while the diagonal value measures failure of complementarity and energy balance.

The first lemma translates the global measure best-response inequality into the familiar local stability and polar-complementarity conditions. Point-mass competitors detect violations of the yield bound, while equality in the polar estimate identifies the direction of the rate measure.

\begin{lemma}[Measure equilibrium and complementarity]
\label{lem:linear-measure-complementarity}
Let $\mu\in\mathcal Y$, let $u=R\mu$, and define the driving force as ${f}_\mu:=h-S\widehat u=-g_\mu$. Then,
\begin{equation}
\label{eq:linear-measure-diagonal-equilibrium}
\mathcal J(\mu;\mu)
\leq
\mathcal J(\nu;\mu) ,
\quad
\forall\nu\in\mathcal Y ,
\end{equation}
is equivalent to
\begin{equation}
\label{eq:linear-measure-complementarity}
\|{f}_\mu(t)\|_H
\leq
\alpha ,
\quad
\forall t\in[0,T] ;
\quad
{f}_\mu
=
\alpha\frac{d\mu}{d|\mu|} ,
\quad
|\mu|\text{-a.e.}
\end{equation}
Under either condition, 
\begin{equation}
\label{eq:linear-local-zero-payoff}
\mathcal J_A(\mu;\mu)=0, \quad \forall A\in \mathcal{B}([0,T]), 
\end{equation}
and therefore
\begin{equation}
\label{eq:linear-energy-identity}
\mathcal E(u(t),t)
+
\alpha|\mu|([0,t])
=
\mathcal E(u_0,0)
-
\int_0^t(\dot h(s),u(s))_H\,ds ,
\end{equation}
for every $t\in[0,T]$.
\end{lemma}

\begin{proof}
Suppose first that \eqref{eq:linear-measure-diagonal-equilibrium} holds. If $\|g_\mu(t_0)\|_H>\alpha$ at some $t_0$, competitors of the form $\nu_r:=-r\frac{g_\mu(t_0)}{\|g_\mu(t_0)\|_H}\delta_{t_0}$, $r>0$, make $\mathcal J(\nu_r;\mu)\to-\infty$, a contradiction. Hence $\|{f}_\mu\|_H\leq\alpha$ pointwise, and therefore $\mathcal J(\nu;\mu)\geq0$ for every $\nu$. Choosing $\nu=0$ in \eqref{eq:linear-measure-diagonal-equilibrium} gives $\mathcal J(\mu;\mu)\leq0$, while stability gives the reverse inequality. Thus, $\mathcal J(\mu;\mu)=0$. The equality case in the polar estimate
\begin{equation} \label{97nuw9}
\int_{[0,T]}({f}_\mu,d\mu)_H
\leq
\alpha|\mu|([0,T])
\end{equation}
gives the second relation in \eqref{eq:linear-measure-complementarity}. The same argument on every Borel restriction gives \eqref{eq:linear-local-zero-payoff}.

Conversely, \eqref{eq:linear-measure-complementarity} gives $\mathcal J(\mu;\mu)=0$ and $\mathcal J(\nu;\mu)\geq0$ for every $\nu$, hence \eqref{eq:linear-measure-diagonal-equilibrium}. Finally, \eqref{eq:linear-energy-identity} follows from \eqref{eq:linear-local-zero-payoff} and \eqref{eq:linear-diagonal-energy-representation}.
\end{proof}

The equilibrium characterization alone does not yet give compactness. The next estimate uses the energy identity to bound both the state and the total variation of the rate by the data, thereby placing all equilibria in one weak-star compact ball, cf.~\cite{MielkeRossiSavare2012}.

\begin{lemma}[Energy-derived measure bound]
\label{lem:linear-energy-measure-bound}
Let $\mu\in\mathcal Y$, set $u=R\mu$, and suppose that \eqref{eq:linear-energy-identity} holds. Define
\begin{equation}
\label{eq:linear-measure-data-constants}
C_\infty
:=
\|h\|_{L^\infty(0,T;H)},
\quad
C_1
:=
\|\dot h\|_{L^1(0,T;H)},
\quad
C_0
:=
\mathcal E(u_0,0)+\frac{C_\infty^2}{\lambda}.
\end{equation}
Then, $C_0\geq0$ and
\begin{equation}
\label{eq:linear-measure-state-bound}
\sup_{t\in[0,T]}\|u(t)\|_H
\leq
{U}
:=
\frac{2}{\lambda}
\left(
C_1+\sqrt{C_1^2+\lambda C_0}
\right).
\end{equation}
Moreover,
\begin{equation}
\label{eq:linear-measure-TV-bound}
|\mu|([0,T])
=
\operatorname{Var}_H(u;[0,T])
\leq
{L}
:=
\frac{C_0+C_1{U}}{\alpha}.
\end{equation}
Thus, every equilibrium rate lies in the weak-star compact ball
\begin{equation}
\label{eq:linear-measure-compact-ball}
\mathcal Y_{{L}}
:=
\left\{
\mu\in\mathcal Y:
|\mu|([0,T])\leq {L}
\right\}.
\end{equation}
\end{lemma}

\begin{proof}
Coercivity and Young's inequality give
\begin{equation}
\label{eq:linear-energy-lower-bound}
\mathcal E(\xi,t)
\geq
\frac{\lambda}{4}\|\xi\|_H^2
-
\frac{C_\infty^2}{\lambda}.
\end{equation}
Let $U_\mu:=\sup_{t\in[0,T]}\|u(t)\|_H$. Combining \eqref{eq:linear-energy-identity} and \eqref{eq:linear-energy-lower-bound}, and discarding the nonnegative dissipation term, yields
\begin{equation}
\frac{\lambda}{4}\|u(t)\|_H^2
\leq
C_0+C_1U_\mu.
\end{equation}
Taking the supremum over $t$ and solving the resulting quadratic inequality gives $U_\mu\leq U$, which is \eqref{eq:linear-measure-state-bound}. Returning to the energy identity and using \eqref{eq:linear-energy-lower-bound} gives
\begin{equation}
\alpha|\mu|([0,T])
\leq
C_0+C_1U,
\end{equation}
which is \eqref{eq:linear-measure-TV-bound}.
\end{proof}

We introduce viscous regularization as a means of generating approximate equilibria. This is consistent with the broader use of vanishing-viscosity approximations in rate-independent evolution \cite{MielkeRossiSavare2012, MielkeRossiSavare2016}, although here the regularization is used specifically to build almost equilibria for the diagonal game. For $\varepsilon>0$, we define the extended measure dissipation as
\begin{equation}
\label{eq:linear-viscous-potential}
\mathcal D_\varepsilon(\nu)
:=
\begin{cases}
\displaystyle
\alpha \| w \|_{L^1} 
+
\frac{\varepsilon}{2} \| w \|_{L^2}^2 , 
&
\nu=w(t)dt,\quad w\in L^2(0,T;H),
\\
+\infty,
&
\text{otherwise} ,
\end{cases}
\end{equation}
and the extended diagonal response and diagonal defect as
\begin{equation}
\label{eq:linear-viscous-payoff}
\mathcal J_\varepsilon(\nu;\mu)
:=
\mathcal D_\varepsilon(\nu)
+
\int_{[0,T]}(g_\mu,d\nu)_H ,
\quad
\mathcal I_\varepsilon(\mu)
:=
\mathcal J_\varepsilon(\mu;\mu)
-
\inf_{\nu\in\mathcal Y}\mathcal J_\varepsilon(\nu;\mu).
\end{equation}
The effective domain of $\mathcal I_\varepsilon$ consists of measures with $L^2$ densities. Thus, $L^2(0,T;H)$ represents the viscous core of the measure space. The following lemma shows that the regularized defect is explicit and, crucially, still has minimum value zero.

\begin{lemma}[Viscous regularization]
Let $\mu=v(t)dt$ with $v\in L^2(0,T;H)$. Then,
\begin{equation}
\label{eq:linear-viscous-defect-representation}
\begin{split}
\mathcal I_\varepsilon(vdt)
& =
\alpha\|v\|_{L^1}
+
\frac{\varepsilon}{2}\|v\|_{L^2}^2
+
\frac12(SKv(T),Kv(T))_H
\\ & +
\int_0^T(Su_0-h(t),v(t))_H\,dt
+
\frac{1}{2\varepsilon}
\int_0^T
\bigl(\|SR(vdt)(t)-h(t)\|_H-\alpha\bigr)_+^2\,dt .
\end{split}
\end{equation}
The extended diagonal defect $\mathcal I_\varepsilon$ has a unique minimizer $\mu_\varepsilon=v_\varepsilon dt$ and
\begin{equation}
\label{eq:linear-viscous-zero-defect}
\mathcal I_\varepsilon(\mu_\varepsilon)=0.
\end{equation}
The density satisfies the inclusion
\begin{equation}
\label{eq:linear-viscous-flow-rule}
0
\in
\alpha\partial\|\cdot\|_H(v_\varepsilon(t))
+
\varepsilon v_\varepsilon(t)
+
S R\mu_\varepsilon(t)-h(t) ,
\quad
\text{for a.e. }t \in (0,T) .
\end{equation}
\end{lemma}

\begin{proof}
If $\mu=v(t)dt$ with $v\in L^2(0,T;H)$, then $R\mu$ is continuous and $g_\mu=S R\mu-h$. Pointwise convex duality, in the standard integral-functional form \cite{Rockafellar1971IntegralII,Rockafellar1971ConvexIntegralDuality}, gives
\begin{equation}
\label{eq:linear-viscous-conjugate}
\inf_{\nu\in\mathcal Y}\mathcal J_\varepsilon(\nu;\mu)
=
-
\frac{1}{2\varepsilon}
\int_0^T
\bigl(\|g_\mu(t)\|_H-\alpha\bigr)_+^2\,dt ,
\end{equation}
whence \eqref{eq:linear-viscous-defect-representation} follows. As a functional of the density, $\mathcal I_\varepsilon$ is finite, coercive, weakly lower semicontinuous, and $\varepsilon$-strongly convex on $L^2(0,T;H)$. Hence it has a unique minimizer $\mu_\varepsilon=v_\varepsilon dt$.

We show that the minimum value is zero. View $\mathcal D_\varepsilon$ on its $L^2$ density core and let $\mathcal D_\varepsilon^*$ denote its $L^2$ convex conjugate. Let $K$ be the Volterra operator on densities, set $A:=SK$, and define $b_\varepsilon(v):=D\mathcal D_\varepsilon^*(-g_{vdt})$. Then, $-g_{vdt}\in\partial\mathcal D_\varepsilon(b_\varepsilon(v))$. The Fermat condition at $v_\varepsilon$ gives $\xi_\varepsilon\in\partial\mathcal D_\varepsilon(v_\varepsilon)$ such that
\begin{equation}
0
=
\xi_\varepsilon
+
g_{\mu_\varepsilon}
+
A^*(v_\varepsilon-b_\varepsilon(v_\varepsilon)).
\end{equation}
Pairing with $z:=v_\varepsilon-b_\varepsilon(v_\varepsilon)$ and using the $\varepsilon$-strong monotonicity of $\partial\mathcal D_\varepsilon$ together with $\langle Az,z\rangle=\frac12(SKz(T),Kz(T))_H\geq0$ gives
\begin{equation}
0
\geq
\varepsilon\|z\|_{L^2}^2
+
\frac12(SKz(T),Kz(T))_H.
\end{equation}
Thus, $z=0$, \eqref{eq:linear-viscous-zero-defect} holds, and \eqref{eq:linear-viscous-flow-rule} follows.
\end{proof}

To pass from the viscous core back to the measure space, we use a temporal Galerkin reduction on entire rate trajectories rather than an incremental time-stepping construction. The approximation below is elementary, but it is designed to preserve total-variation control while being strongly dense on the $L^2$ core. For $m\in\mathbb N$, let
\begin{equation}
\label{eq:linear-galerkin-grid}
N_m
:=
2^m,
\quad
\tau_m
:=
\frac{T}{N_m},
\quad
0=t_0^m<t_1^m<\cdots<t_{N_m}^m=T,
\quad
I_i^m
:=
(t_{i-1}^m,t_i^m) .
\end{equation}
We then define the nested temporal measure-approximation subspaces
\begin{equation}
\label{eq:linear-galerkin-spaces}
\mathcal Y_m
:=
\left\{
\left(\sum_{i=1}^{N_m}\eta_i\chi_{I_i^m}\right)dt:
\eta_i\in H
\right\}
\subset
\mathcal Y ,
\end{equation}
As in the superlinear construction, only time is discretized: the coefficients $\eta_i$ remain in $H$, and no spatial finite-dimensional approximation is implied. We denote by
\begin{equation}
\label{eq:linear-interval-projection}
(P_mf)(t)
:=
\frac{1}{\tau_m}
\int_{I_i^m}f(s)\,ds
\quad
\text{for }t\in I_i^m 
\end{equation}
the interval-averaging projection.

These approximation spaces have the following properties. The weak-star part is a concrete measure approximation, while the strong $L^r$ part is the usual convergence of interval averages; cf.~standard vector-measure and functional-analysis treatments \cite{DiestelUhl1977,Brezis2011}.

\begin{lemma}[Approximation measure spaces]
\label{lem:linear-approximation-measure-spaces}
The union $\bigcup_{m\in\mathbb N}\mathcal Y_m$ is weak-star sequentially dense in $\mathcal Y$. More precisely, for every $\mu\in\mathcal Y$ there exists a sequence $\mu_m\in\mathcal Y_m$ such that $\mu_m\stackrel{*}{\rightharpoonup}\mu$ in $\mathcal Y$ and $|\mu_m|([0,T])\leq |\mu|([0,T])$. Moreover, for every $1\leq r<\infty$ and every $f\in L^r(0,T;H)$, $P_mf\to f$ strongly in $L^r(0,T;H)$. In particular, after identifying $v(t)dt$ with its density $v$, the union $\bigcup_{m\in\mathbb N}\mathcal Y_m$ is strongly dense in the $L^2$-density core of $\mathcal Y$.
\end{lemma}

\begin{proof}
We first prove the weak-star approximation of an arbitrary measure. For $i=1,\ldots,N_m$, introduce the Borel partition
\begin{equation}
B_i^m
:=
\begin{cases}
[t_{i-1}^m,t_i^m),
& i=1,\ldots,N_m-1,
\\
[t_{N_m-1}^m,T],
& i=N_m.
\end{cases}
\end{equation}
For $\mu\in\mathcal Y$, define
\begin{equation}
\mu_m
:=
\Big(
\sum_{i=1}^{N_m}
\frac{\mu(B_i^m)}{\tau_m}\chi_{I_i^m}
\Big)dt.
\end{equation}
Then, $\mu_m\in\mathcal Y_m$, and
\begin{equation}
|\mu_m|([0,T])
=
\sum_{i=1}^{N_m}\|\mu(B_i^m)\|_H
\leq
\sum_{i=1}^{N_m}|\mu|(B_i^m)
=
|\mu|([0,T]).
\end{equation}
Let $\varphi\in C([0,T];H)$ and set
\begin{equation}
\overline\varphi_i^m
:=
\frac{1}{\tau_m}\int_{I_i^m}\varphi(s)\,ds,
\quad
\omega_\varphi(\delta)
:=
\sup_{\substack{s,t\in[0,T]\\|s-t|\leq\delta}}
\|\varphi(s)-\varphi(t)\|_H.
\end{equation}
Since every $s\in I_i^m$ and $t\in B_i^m$ satisfy $|s-t|\leq\tau_m$, we have
\begin{equation}
\|\overline\varphi_i^m-\varphi(t)\|_H
\leq
\omega_\varphi(\tau_m) ,
\quad
\forall t\in B_i^m.
\end{equation}
Therefore,
\begin{equation}
\bigl|
\int_{[0,T]}(\varphi,d\mu_m)_H
-
\int_{[0,T]}(\varphi,d\mu)_H
\bigr|
=
\bigl|
\sum_{i=1}^{N_m}
\int_{B_i^m}
(\overline\varphi_i^m-\varphi(t),d\mu(t))_H
\bigr|
\leq 
\omega_\varphi(\tau_m)|\mu|([0,T]).
\end{equation}
Uniform continuity of $\varphi$ and $\tau_m\to0$ show that the right-hand side tends to zero. Hence, $\mu_m\stackrel{*}{\rightharpoonup}\mu$ in $\mathcal Y$.

Next, we prove the strong convergence of the interval averages. For every $f\in L^r(0,T;H)$, Jensen's inequality gives
\begin{equation}
\|P_mf\|_{L^r}^r
=
\sum_{i=1}^{N_m}
\tau_m
\bigl\|
\frac{1}{\tau_m}\int_{I_i^m}f(s)\,ds
\bigr\|_H^r
\leq 
\sum_{i=1}^{N_m}
\int_{I_i^m}\|f(s)\|_H^r\,ds
=
\|f\|_{L^r}^r.
\end{equation}
Thus, $P_m$ is a contraction on $L^r(0,T;H)$. If $\varphi\in C([0,T];H)$, then for almost every $t\in I_i^m$,
\begin{equation}
\|(P_m\varphi)(t)-\varphi(t)\|_H
\leq
\frac{1}{\tau_m}
\int_{I_i^m}\|\varphi(s)-\varphi(t)\|_H\,ds
\leq
\omega_\varphi(\tau_m) ,
\end{equation}
whence it follows that
\begin{equation}
\|P_m\varphi-\varphi\|_{L^r}
\leq
T^{1/r}\omega_\varphi(\tau_m)
\to0.
\end{equation}
Since $C([0,T];H)$ is dense in $L^r(0,T;H)$, for every $f\in L^r(0,T;H)$ and every $\delta>0$ we may choose $\varphi\in C([0,T];H)$ such that $\|f-\varphi\|_{L^r}<\delta$. The contraction property then yields
\begin{equation}
\|P_mf-f\|_{L^r}
\leq
\|P_m(f-\varphi)\|_{L^r}
+
\|P_m\varphi-\varphi\|_{L^r}
+
\|\varphi-f\|_{L^r}
\leq 
2\delta
+
\|P_m\varphi-\varphi\|_{L^r}.
\end{equation}
Taking the limsup as $m\to\infty$ and then letting $\delta\downarrow0$ proves $P_mf\to f$ strongly in $L^r(0,T;H)$. Finally, if $\mu=v(t)dt$ with $v\in L^2(0,T;H)$, then $(P_mv)(t)dt\in\mathcal Y_m$ and $P_mv\to v$ strongly in $L^2(0,T;H)$, which proves strong density in the $L^2$-density core.
\end{proof}

We now have the three ingredients needed for the linear-growth construction: an energy-derived compact ball, exact equilibria of the viscous core, and temporal measure-approximation spaces that approximate both measures and $L^2$ densities. The next proposition combines them and then uses the regularity of the loading to show that the limiting measure is in fact absolutely continuous.

\begin{proposition}[Direct existence, uniqueness, and temporal measure approximation]
\label{prop:linear-measure-attainment}
Under assumptions \eqref{eq:linear-measure-coercivity}--\eqref{eq:linear-measure-initial-stability}, there exists a unique $\mu^*\in\mathcal Y$ satisfying the measure diagonal equilibrium
\begin{equation}
\label{eq:linear-main-measure-equilibrium}
\mathcal J(\mu^*;\mu^*)
\leq
\mathcal J(\nu;\mu^*) ,
\quad
\forall \nu\in\mathcal Y.
\end{equation}
The state $u^*:=R\mu^*$ belongs to $\mathcal X$, satisfies \eqref{eq:linear-measure-complementarity} and \eqref{eq:linear-energy-identity}, and $\mu^*\in\mathcal Y_{{L}}$, see~(\ref{eq:linear-measure-compact-ball}). Moreover, the measure equilibrium is regular: $\mu^* = v^*(t)\,dt$ for some $v^*\in L^1(0,T;H)$, $u^*\in W^{1,1}(0,T;H)$, and
\begin{equation}
\label{eq:linear-measure-flow-rule}
0
\in
\alpha\partial\|\cdot\|_H(v^*(t))
+
Su^*(t)-h(t) ,
\quad
\text{for a.e. }t\in(0,T).
\end{equation}
Consequently, $u^*$ is the unique energetic solution in the sense of Mielke and Theil \cite{MielkeTheil2004}: for every $t\in[0,T]$,
\begin{equation}
\label{eq:linear-energetic-stability}
\mathcal E(u^*(t),t)
\leq
\mathcal E(z,t)+\alpha\|z-u^*(t)\|_H ,
\quad
\forall z\in H,
\end{equation}
and \eqref{eq:linear-energy-identity} holds with $\mu^*=v^*dt$. In addition, for every fixed reference total-variation scale $L_0>0$ with the same physical units as $L$, there is a sequence $\mu_n=v_n(t)dt$ such that $\mu_n \stackrel{*}{\rightharpoonup} \mu^*$ in $\mathcal Y$,
\begin{equation}
\label{eq:linear-galerkin-residual-convergence}
|\mu_n|([0,T])
\leq
L+\frac{L_0}{n},
\quad
\mathcal Q(\mu_n)
\to
0,
\quad
\mathcal J(\mu_n;\mu_n)
\to
0,
\quad
\mathcal I_M(\mu_n)
\to
0,
\quad
\forall M>0,
\end{equation}
with $v_n \rightharpoonup v^*$ in $L^1(0,T;H)$.
\end{proposition}

\begin{proof}
We divide the proof into restricted temporal minimization of the regularized measure problem, measure compactness, and identification of the limit.

\smallskip
\noindent
\emph{Step 1: Approximate equilibria.}
For fixed $\varepsilon>0$, let $\mu_{m,\varepsilon}\in\operatorname{argmin}_{\mu\in\mathcal Y_m}\mathcal I_\varepsilon(\mu)$. This is the measure version of the pathwise restricted minimization used in Proposition~\ref{prop:superlinear-attainment}. Writing $\mu_{m,\varepsilon}=v_{m,\varepsilon}dt$, restricted minimality and density give
\begin{equation} \label{fxqEb8}
0
\leq
\mathcal I_\varepsilon(\mu_{m,\varepsilon})
\leq
\mathcal I_\varepsilon((P_mv_\varepsilon)dt)
\to
\mathcal I_\varepsilon(\mu_\varepsilon)
=
0.
\end{equation}
Since $\mathcal I_\varepsilon$ is $\varepsilon$-strongly convex and $\mu_\varepsilon$ is its global minimizer,
\begin{equation} \label{HXQWpV}
\frac{\varepsilon}{2}
\|v_{m,\varepsilon}-v_\varepsilon\|_{L^2}^2
\leq
\mathcal I_\varepsilon(\mu_{m,\varepsilon})
-
\mathcal I_\varepsilon(\mu_\varepsilon).
\end{equation}
Consequently, $v_{m,\varepsilon} \to v_\varepsilon$ strongly in $L^2(0,T;H)$ as $m\to\infty$.

\smallskip
\noindent
\emph{Step 2: Uniform estimates in the viscosity.}
Set $u_\varepsilon:=R\mu_\varepsilon$. Pairing \eqref{eq:linear-viscous-flow-rule} with $v_\varepsilon$ and applying the ordinary chain rule gives
\begin{equation}
\label{eq:linear-viscous-energy-identity}
\mathcal E(u_\varepsilon(t),t)
+
\alpha\int_0^t\|v_\varepsilon(s)\|_H\,ds
+
\varepsilon\int_0^t\|v_\varepsilon(s)\|_H^2\,ds
= 
\mathcal E(u_0,0)
-
\int_0^t(\dot h(s),u_\varepsilon(s))_H\,ds ,
\end{equation}
for every $t\in[0,T]$. The proof of Lemma~\ref{lem:linear-energy-measure-bound}, with the additional nonnegative viscous term discarded, gives
\begin{equation}
\label{eq:linear-viscous-energy-bounds}
\sup_{t\in[0,T]}\|u_\varepsilon(t)\|_H
\leq
{U},
\quad
|\mu_\varepsilon|([0,T])
=
\|v_\varepsilon\|_{L^1}
\leq
{L}.
\end{equation}
The flow rule is equivalent to
\begin{equation}
\label{eq:linear-viscous-projection-relation}
{f}_\varepsilon-P_{C_\alpha}{f}_\varepsilon
=
\varepsilon v_\varepsilon ,
\quad
{f}_\varepsilon
:=
h-Su_\varepsilon,
\quad
C_\alpha
:=
\{\zeta\in H:\|\zeta\|_H\leq\alpha\} ,
\end{equation}
where $P_{C_\alpha}$ is the nearest-point projection from $H$ onto the closed ball $C_\alpha$. Set $r_\varepsilon(t):=\operatorname{dist}_H({f}_\varepsilon(t),C_\alpha)$. Since the squared distance to a closed convex set is differentiable and $\dot{f}_\varepsilon=\dot h-Sv_\varepsilon$, we obtain
\begin{equation}
\frac{d}{dt}\frac12r_\varepsilon(t)^2
\leq
r_\varepsilon(t)\|\dot h(t)\|_H
-
\frac{\lambda}{\varepsilon}r_\varepsilon(t)^2 ,
\quad
\text{for a.e. }t ,
\end{equation}
with $\lambda$ as in (\ref{eq:linear-measure-coercivity}). The initial stability condition (\ref{eq:linear-measure-initial-stability}) gives $r_\varepsilon(0)=0$. Hence, by Gr\"onwall's inequality,
\begin{equation}
\label{eq:linear-viscous-overstress-estimate}
r_\varepsilon(t)
\leq
\int_0^t
{\rm e}^{-\lambda(t-s)/\varepsilon}
\|\dot h(s)\|_H\,ds .
\end{equation}
Taking norms in (\ref{eq:linear-viscous-projection-relation}) gives $r_\epsilon(t) = \varepsilon \| v_\varepsilon(t) \|_H$, whence  (\ref{eq:linear-viscous-overstress-estimate}) becomes
\begin{equation}
\label{eq:linear-viscous-rate-domination}
\|v_\varepsilon(t)\|_H
\leq
\frac1\varepsilon
\int_0^t
{\rm e}^{-\lambda(t-s)/\varepsilon}
\|\dot h(s)\|_H\,ds.
\end{equation}
The absolute continuity of the $L^1$ integral implies $\mathcal Q(\mu_\varepsilon) = \|r_\varepsilon\|_{L^\infty(0,T)} \to 0$. To make the compactness consequence explicit, extend $a(t):=\|\dot h(t)\|_H$ by zero outside $(0,T)$ and set
\begin{equation} \label{mQgj5N}
k_\varepsilon(r)
:=
\frac1\varepsilon
{\rm e}^{-{\lambda r}/{\varepsilon}}
\mathbf 1_{[0,\infty)}(r).
\end{equation}
The right-hand side of \eqref{eq:linear-viscous-rate-domination} is $(k_\varepsilon*a)(t)$. Since $\lambda k_\varepsilon$ is an approximate identity, it follows that $k_\varepsilon*a \to a/\lambda$ strongly in $L^1(\mathbb R)$. Hence $(k_\varepsilon*a)$, and therefore $(v_\varepsilon)$, is uniformly integrable in $L^1(0,T;H)$. Moreover,
\begin{equation}
\label{eq:linear-viscous-dissipation-convergence}
0
\leq
\varepsilon\|v_\varepsilon\|_{L^2}^2
=
\int_0^T r_\varepsilon(t)\|v_\varepsilon(t)\|_H\,dt
\leq
\mathcal Q(\mu_\varepsilon){L}
\to
0.
\end{equation}
Pairing \eqref{eq:linear-viscous-flow-rule} with $v_\varepsilon$ also gives
\begin{equation}
\label{eq:linear-viscous-unregularized-value}
\mathcal J(\mu_\varepsilon;\mu_\varepsilon)
=
-\varepsilon\|v_\varepsilon\|_{L^2}^2
\to
0.
\end{equation}

\smallskip
\noindent
\emph{Step 3: Diagonal temporal approximation and measure compactness.}
Fix $L_0$ as in the statement and choose any sequence $\varepsilon_n\downarrow0$. By (\ref{fxqEb8}-\ref{HXQWpV}) and the continuity of $\mathcal J$ and $\mathcal Q$ under strong $L^2$-convergence, we may choose a strictly increasing sequence $m_n$ such that, with $\mu_n:=\mu_{m_n,\varepsilon_n}=v_n\,dt$, we have
\begin{equation}
\label{eq:linear-diagonal-galerkin-accuracy}
\mathcal I_{\varepsilon_n}(\mu_n) \lesssim \frac{1}{n} ,
\quad
\left|
\mathcal J(\mu_n;\mu_n)
-
\mathcal J(\mu_{\varepsilon_n};\mu_{\varepsilon_n})
\right| \lesssim \frac{1}{n} ,
\quad
|\mathcal Q(\mu_n)-\mathcal Q(\mu_{\varepsilon_n})|
\lesssim \frac{1}{n} ,
\end{equation}
and 
\begin{equation} \label{FWJp2a}
    \|v_n-v_{\varepsilon_n}\|_{L^1}\leq \frac{L_0}{n} .
\end{equation}
From (\ref{eq:linear-viscous-energy-bounds}) and (\ref{FWJp2a}) we therefore obtain
\begin{equation} \label{ZyY5Va}
|\mu_n|([0,T])
=
\|v_n\|_{L^1(0,T;H)} 
\leq 
\|v_{\varepsilon_n}\|_{L^1(0,T;H)}
+ 
\|v_n-v_{\varepsilon_n}\|_{L^1(0,T;H)} 
\leq 
L+\frac{L_0}{n}.
\end{equation}
Equations \eqref{eq:linear-viscous-rate-domination}, \eqref{eq:linear-viscous-unregularized-value}, and \eqref{eq:linear-diagonal-galerkin-accuracy} yield the limits for $\mathcal Q(\mu_n)$ and $\mathcal J(\mu_n;\mu_n)$ in \eqref{eq:linear-galerkin-residual-convergence}. Formula \eqref{eq:linear-bounded-defect-representation} yields the limit for $\mathcal I_M(\mu_n)$.

The bound (\ref{ZyY5Va}) places $(\mu_n)$ in the fixed weak-star compact ball $\mathcal Y_{L+L_0}$. Hence, along a subsequence, $\mu_n \stackrel{*}{\rightharpoonup} \mu^*$ in $\mathcal Y$. This is the primary compactness passage. In addition, \eqref{eq:linear-diagonal-galerkin-accuracy} makes $(v_n)$ an $L^1$-vanishing perturbation of the uniformly integrable family $(v_{\varepsilon_n})$. The vector-valued Dunford--Pettis compactness criterion \cite{DiestelUhl1977} therefore gives, after taking a further subsequence, $v_n \rightharpoonup v^*$ in $L^1(0,T;H)$. The weak-star limit is, consequently, $\mu^* = v^*dt$. We note that absolute continuity is obtained as a regularity conclusion after measure compactness, and is not assumed in the solution space.

\smallskip
\noindent
\emph{Step 4: Identification of the measure equilibrium.}
Set $u^*:=R\mu^*$. For every $t\in[0,T]$, weak $L^1$ convergence gives
\begin{equation}
\int_0^t v_n(s)\,ds
\rightharpoonup
\int_0^t v^*(s)\,ds
\quad
\text{in }H.
\end{equation}
Since $\mathcal Q(\mu_n)\to0$, weak lower semicontinuity of the norm gives
\begin{equation}
\label{eq:linear-yield-condition}
\|Su^*(t)-h(t)\|_H
\leq
\alpha ,
\quad
\forall t\in[0,T].
\end{equation}
For the absolutely continuous measures in the approximating sequence, the diagonal payoff has the Volterra representation
\begin{equation}
\mathcal J(\mu_n;\mu_n)
=
\alpha\|v_n\|_{L^1}
+
\frac12
\bigl(
S\int_0^T v_n(s)\,ds,
\int_0^T v_n(s)\,ds
\bigr)_H
+
\int_0^T\bigl(Su_0-h(t),v_n(t)\bigr)_H\,dt.
\end{equation}
Under weak $L^1$ convergence, the first term and the positive quadratic term are weakly lower semicontinuous, while the last term is weakly continuous. Therefore,
\begin{equation}
\mathcal J(\mu^*;\mu^*)
\leq
\liminf_{n\to\infty}
\mathcal J(\mu_n;\mu_n)
=
0.
\end{equation}
The safe-stress condition (\ref{eq:linear-measure-complementarity}) makes the diagonal integrand nonnegative, whence we conclude that $\mathcal J(\mu^*;\mu^*)=0$. Lemma~\ref{lem:linear-measure-complementarity} now gives \eqref{eq:linear-main-measure-equilibrium}, the measure complementarity relation, and the energy identity. Since $\mu^*=v^*dt$, complementarity is exactly the flow rule \eqref{eq:linear-measure-flow-rule}. Expanding the quadratic energy shows that the safe-stress condition is equivalent to \eqref{eq:linear-energetic-stability}. Lemma~\ref{lem:linear-energy-measure-bound} gives $\mu^*\in\mathcal Y_L$.

\smallskip
\noindent
\emph{Step 5: Uniqueness in the measure space.}
Let $\mu_1,\mu_2\in\mathcal Y$ be two measure equilibria, set $u_i:=R\mu_i$, $z:=u_1-u_2$, and ${f}_i:=h-S\widehat u_i$. Since $\mu_i$ are measure equilibria, (\ref{eq:linear-main-measure-equilibrium}) and Lemma~\ref{lem:linear-measure-complementarity} yield the complementarity and stability relations (\ref{eq:linear-measure-complementarity}). Hence, using the localized polar estimate (\ref{97nuw9}), for every $t \in [0,T]$ we obtain
\begin{equation}
\int_{[0,t]}({f}_i,d\mu_i)_H
=
\alpha|\mu_i|([0,t]),
\quad
\int_{[0,t]}({f}_i,d\mu_j)_H
\leq
\alpha|\mu_j|([0,t]).
\end{equation}
Therefore,
\begin{equation}
0
\leq
\int_{[0,t]}({f}_1-{f}_2,d(\mu_1-\mu_2))_H
=
-
\int_{[0,t]}(S\widehat z,dDz)_H.
\end{equation}
The centered chain rule (\ref{eq:linear-quadratic-measure-chain-rule}) then gives
\begin{equation}
0
\leq
-
\frac12(Sz(t),z(t))_H.
\end{equation}
Coercivity implies $z(t)=0$ for every $t$, hence $\mu_1=\mu_2$. Every convergent subsequence of the temporal approximation family therefore has the same limit, so the full sequence satisfies $\mu_n \stackrel{*}{\rightharpoonup} \mu^*$ in $\mathcal Y$ and $v_n \rightharpoonup v^*$ in $L^1(0,T;H)$.
\end{proof}

\begin{remark}[Recovery of energetic solutions]
{\rm In the convex rate-independent setting, the diagonal equilibrium recovers the energetic solution of Mielke and Theil \cite{MielkeTheil2004}, whereas the whole-trajectory formulation is designed to remain meaningful when non-convex and state/branch-dependent kinetics obstruct the restart of an incremental construction.}\hfill$\square$
\end{remark}

\begin{remark}[Relation with the abstract criteria]
{\rm Unlike the superlinear case of Proposition~\ref{prop:superlinear-attainment}, Proposition~\ref{aBgxdV} is not directly applicable here. Although the diagonal term is weak-star lower semicontinuous through the energy representation~(\ref{eq:linear-diagonal-energy-representation}), for a fixed atomic competitor $\nu$ the map $\mu \mapsto J(\nu;\mu)$ need not be weak-star continuous, owing to the discontinuity of the centered Volterra kernel on the time diagonal. The proof instead uses the stability residual $\mathcal Q$ and energy-derived measure compactness, and then exploits the $W^{1,1}$ loading estimate to obtain an absolutely continuous limit before recovering the full best-response relation.}\hfill$\square$
\end{remark}

\begin{remark}[Rough loading] \label{ZhGtXu}
{\rm The measure payoff \eqref{eq:linear-measure-payoff}, the centered chain rule, and the energy-derived total-variation bound remain meaningful beyond absolutely continuous loading. If $h$ is merely of bounded variation, the weak-star measure compactness still survives, but the uniform-integrability estimate \eqref{eq:linear-viscous-rate-domination} need not. The limiting rate may then retain atoms or a singular continuous part. Work at jumps of the loading must be incorporated through a completed graph or an equivalent parametrized formulation, as in Balanced-Viscosity theory \cite{MielkeRossiSavare2012, MielkeRossiSavare2016}.}\hfill$\square$
\end{remark}

\section{A game-theoretical relaxation method}
\label{sec:direct-almost-equilibrium-relaxation}

The strategy is to choose an extended ambient space $\mathcal Y$ from the outset, with the pure rate and competitor space $Y$ included in $\mathcal Y$. The pure payoff $J$ and diagonal defect $I$ are first defined on $Y$ and then extended to calligraphic functions $\mathcal J$ and $\mathcal I$ on $\mathcal Y$ by assigning infinite cost outside the pure space. Pure recovery sequences therefore converge directly in the space in which relaxed limits live, without an observation map or realization classes. The tradeoff is that the infinite extensions do not define a meaningful pointwise game at genuinely relaxed points. The relaxed defect is instead the sequential lower-semicontinuous envelope $\overline{\mathcal I}$ of $\mathcal I$. This is a relaxation of the \emph{equilibrium defect}, not a claim that the payoff itself admits a canonical pointwise extension to relaxed states.

Let ${Y}$ be the set of pure rates and competitors, and let $J:{Y}\times {Y}\to(-\infty,+\infty]$. The finite-cost diagonal class is
\begin{equation} \label{Mptsra}
D(J):=\{v\in {Y}:J(v;v)<+\infty\},
\end{equation}
and is assumed to be nonempty. Competitors range over all of ${Y}$, whereas finite-defect recovery sequences take values in $D(J)$. Define the pure diagonal defect on all of ${Y}$ by
\begin{equation}
\label{eq:relaxation-diagonal-defect}
I(v)
:=
\begin{cases}
\displaystyle
\sup{}_{w\in {Y}}
\bigl\{
J(v;v)-J(w;v)
\bigr\},&v\in D(J),\\
+\infty,&v\in {Y}\setminus D(J).
\end{cases}
\end{equation}
Thus, $I:{Y}\to[0,+\infty]$. If $v\in D(J)$ and $m(v):=\inf_{w\in {Y}}J(w;v)>-\infty$, then $I(v)=J(v;v)-m(v)$. When $m(v)=-\infty$, the convention in \eqref{eq:relaxation-diagonal-defect} gives $I(v)=+\infty$. Thus, no artificial bounded-below assumption is imposed on unstable competitor sections. For every $v\in D(J)$ and every $\varepsilon\geq0$, we have
\begin{equation} \label{SVxmYv}
I(v)\leq\varepsilon
\quad\Longleftrightarrow\quad
J(v;v)\leq J(w;v)+\varepsilon ,
\quad
\forall w\in {Y} .
\end{equation}
In particular, zeros of $I$ are pure diagonal equilibria.

Let $(\mathcal Y,d)$ be an enriched metric space containing ${Y}$ as a distinguished subspace. Extend $J$ to $\mathcal J:\mathcal Y\times\mathcal Y\to(-\infty,+\infty]$ by
\begin{equation}
\label{eq:infinite-payoff-extension-formula}
\mathcal J(w;v)
:=
\begin{cases}
J(w;v),&w,v\in {Y},\\
+\infty,&\text{otherwise}.
\end{cases}
\end{equation}
Extend $I$ to the enriched space by
\begin{equation}
\label{eq:infinite-defect-extension}
\mathcal I(z)
:=
\begin{cases}
I(z),&z\in {Y},\\
+\infty,&z\in\mathcal Y\setminus {Y}.
\end{cases}
\end{equation}
Thus, $\mathcal I|_{{Y}}=I$. Its sequential lower-semicontinuous envelope is
\begin{equation}
\label{eq:relaxed-diagonal-defect}
\overline{\mathcal I}(z)
:=
\inf_{\substack{(z_h)\subset\mathcal Y\\ z_h\to z}}
\liminf_{h\to\infty}\mathcal I(z_h)
=
\inf_{\substack{(v_h)\subset D(J)\\ v_h\to z\text{ in }\mathcal Y}}
\liminf_{h\to\infty}I(v_h),
\quad
z\in\mathcal Y,
\end{equation}
with value $+\infty$ when no sequence in $D(J)$ converges to $z$. To justify the second equality, consider a sequence $(z_h)\subset\mathcal Y$ with finite defect liminf and pass to a subsequence realizing that liminf. The values $\mathcal I(z_h)$ are then finite along a further tail, so $z_h\in {Y}$ and $I(z_h)<+\infty$. By the definition of $I$, this implies $z_h\in D(J)$. Conversely, every $D(J)$-valued sequence is admissible in the first infimum.

This observation suggests defining relaxed solutions by what can actually be approximated: pure paths with vanishing diagonal defect. The definition retains the original pure competitor class and records only the enriched-space limit of such almost equilibria.

\begin{definition}[Relaxed diagonal equilibrium]
\label{def:I-relaxed-diagonal-equilibrium}
A point $z^*\in\mathcal Y$ is a relaxed diagonal equilibrium if there exists a sequence $(v_h)\subset D(J)$ such that $v_h\to z^*$ in $\mathcal Y$ and $I(v_h)\to 0$. The set of relaxed diagonal equilibria is denoted by
\begin{equation}
\label{eq:relaxed-equilibrium-set}
\overline{M}(\mathcal J)
:=
\bigr\{
z\in\mathcal Y \,: \,
\exists \, (v_h)\subset D(J) \text{ s.t. }v_h\to z
\text{ in }\mathcal Y 
\text{ and }I(v_h)\to0
\bigl\}.
\end{equation}
\end{definition}

The definition can be expressed entirely in terms of the relaxed defect. The next proposition shows that no information is lost at level zero: taking the sequential lower-semicontinuous envelope and taking the closure of pure almost equilibria give the same relaxed equilibrium set.

\begin{proposition}[Relaxed-defect characterization and losslessness]
\label{prop:abstract-relaxation-equivalence}
The relaxed equilibrium set is the zero set of the relaxed diagonal defect:
\begin{equation}
\label{eq:relaxed-equilibrium-zero-set}
\overline{M}(\mathcal J)
=
\bigr\{
z\in\mathcal Y:
\overline{\mathcal I}(z)=0
\bigl\}.
\end{equation}
Equivalently, $z\in\overline{M}(\mathcal J)$ if and only if there exist $(v_h)\subset D(J)$ and $\varepsilon_h\downarrow0$ such that $v_h\to z$ in $\mathcal Y$ and
\begin{equation}
\label{eq:ambient-recovery-almost-equilibrium}
J(v_h;v_h)
\leq
J(w;v_h)+\varepsilon_h ,
\quad
\forall w\in {Y} .
\end{equation}
By the infinite extension, this is also equivalent to
\begin{equation}
\mathcal J(v_h;v_h)
\leq
\mathcal J(z;v_h)+\varepsilon_h ,
\quad
\forall z\in\mathcal Y .
\end{equation}
Thus, at the zero-defect level, the relaxation is exactly the sequential closure in $\mathcal Y$ of pure diagonal almost equilibria with vanishing error.
\end{proposition}

\begin{proof}
Suppose first that $z\in\overline{M}(\mathcal J)$. By Definition~\ref{def:I-relaxed-diagonal-equilibrium}, there exists $(v_h)\subset D(J)$ such that $v_h\to z$ in $\mathcal Y$ and $I(v_h)\to0$. Since $\mathcal I(v_h)=I(v_h)$, definition \eqref{eq:relaxed-diagonal-defect} gives $0\leq\overline{\mathcal I}(z)\leq\liminf_{h\to\infty}I(v_h)=0$, hence $\overline{\mathcal I}(z)=0$.

Conversely, suppose that $\overline{\mathcal I}(z)=0$. For every $k\in\mathbb N$, definition \eqref{eq:relaxed-diagonal-defect} yields a sequence $(v_{k,n})_{n\in\mathbb N}\subset D(J)$ such that $v_{k,n}\to z$ in $\mathcal Y$ and $\liminf_{n\to\infty}I(v_{k,n})<\frac{1}{4k}$. There exists $n(k)$ such that $d_{\mathcal Y}(v_{k,n(k)},z)<\frac{1}{k}$ and $I(v_{k,n(k)})<\frac{1}{2k}$. Setting $v_k:=v_{k,n(k)}$, we obtain $v_k\to z$ in $\mathcal Y$ and $I(v_k)\to0$. Therefore $z\in\overline{M}(\mathcal J)$.

It remains to prove the almost-equilibrium characterizations. Suppose that $z\in\overline{M}(\mathcal J)$ and let $(v_h)\subset D(J)$ be a recovery sequence. After passing to a subsequence and relabeling, we may assume that $I(v_h)\leq\varepsilon_h\downarrow0$, and \eqref{SVxmYv} gives $J(v_h;v_h)\leq J(w;v_h)+\varepsilon_h$ for all $w\in {Y}$. Conversely, suppose that $(v_h)\subset D(J)$, $\varepsilon_h\downarrow0$, $v_h\to z$ in $\mathcal Y$, and \eqref{eq:ambient-recovery-almost-equilibrium} holds. By \eqref{SVxmYv}, $0\leq I(v_h)\leq\varepsilon_h\to0$, so Definition~\ref{def:I-relaxed-diagonal-equilibrium} gives $z\in\overline{M}(\mathcal J)$.

Finally, because $v_h\in D(J)\subset {Y}$, $\mathcal J(v_h;v_h)=J(v_h;v_h)$. For $z\in {Y}$, one has $\mathcal J(z;v_h)=J(z;v_h)$, whereas for $z\in\mathcal Y\setminus {Y}$ one has $\mathcal J(z;v_h)=+\infty$. The two almost-equilibrium inequalities are therefore equivalent.
\end{proof}

The characterization reduces existence in the enriched space to a compactness question. In particular, once a vanishing-defect family is relatively compact in $\mathcal Y$, any cluster point is automatically a relaxed equilibrium.

\begin{proposition}[Compactness produces a relaxed equilibrium]
\label{prop:relaxed-equilibrium-compactness}
Suppose that there exist $(v_h)\subset D(J)$ and $\varepsilon_h\downarrow0$ satisfying
\begin{equation}
J(v_h;v_h)
\leq
J(w;v_h)+\varepsilon_h ,
\quad
\forall w\in {Y}.
\end{equation}
If $(v_h)$ has a convergent subsequence in $\mathcal Y$, then every subsequential limit is a relaxed diagonal equilibrium. In particular, suppose that for every $\varepsilon>0$ there exists $v_\varepsilon\in D(J)$ satisfying $I(v_\varepsilon) \leq\varepsilon$, and that for some $a>0$ the set $\left\{ v\in D(J)\,:\,I(v)\leq a \right\}$ is relatively sequentially compact in $\mathcal Y$. Then, $\overline{M}(\mathcal J)$ is nonempty.
\end{proposition}

\begin{proof}
Along a convergent subsequence, \eqref{SVxmYv} gives $I(v_h)\leq\varepsilon_h\to0$. Proposition~\ref{prop:abstract-relaxation-equivalence} then applies. For the final assertion, choose $\varepsilon_h\downarrow0$ with $\varepsilon_h\leq a$ and corresponding $v_h\in D(J)$ such that $I(v_h)\leq\varepsilon_h$. The stated relative sequential compactness yields a convergent subsequence, and the first part applies.
\end{proof}

\begin{remark}[The infinite extensions are not a relaxed game] \label{JeyRfv}
{\rm For $z\in\mathcal Y\setminus {Y}$, one has $\mathcal J(z;z)=+\infty$ and $\mathcal J(w;z)=+\infty$ for every $w\in\mathcal Y$. The diagonal inequality at $z$ is therefore vacuous and an expression such as $\mathcal J(z;z)-\inf_w\mathcal J(w;z)$ is not defined. The functions $\mathcal J$ and $\mathcal I$ only embed the pure game and its defect into the ambient space; the genuine relaxed quantity is $\overline{\mathcal I}$.}\hfill$\square$
\end{remark}

\begin{remark}[Distinction from evolutionary $\Gamma$-convergence]
{\rm Convergence of rate-independent evolutions under $\Gamma$-convergence of the underlying energies has been studied, among other settings, by Negri \cite{Negri2014}. The relaxation considered here is of a different type. Rather than starting from a sequence of energy functionals and passing to the associated limiting evolution, we fix the pure evolutionary game and close sequences of pure almost equilibria whose diagonal defects vanish. Thus, the object being relaxed is the whole-trajectory equilibrium principle itself, together with the state--rate information retained by the chosen enriched topology.}\hfill$\square$
\end{remark}

\begin{example}[Linear-growth evolutions as a lossless bounded-competitor relaxation]
\label{ex:linear-growth-relaxation}
{\rm As a first example of relaxation, consider the linear-growth energy--dissipation system of Section~\ref{subsec:linear-growth-energy-dissipation}. It may be viewed first on the pure state and rate spaces
\begin{equation}
X:=\bigl\{u\in W^{1,1}(0,T;H):u(0)=u_0\bigr\},\quad Y:=L^1(0,T;H),
\end{equation}
with reconstruction operator $R$ as in (\ref{7TzojC}), and then embedded in the ambient spaces (\ref{eq:linear-measure-spaces}) by identifying $v\in Y$ with the measure $v(t)\,dt$ and reconstructing the state by $u=R\mu$ as in \eqref{eq:linear-measure-reconstruction}. The pure functionals are extended by $+\infty$ outside $X$ and $Y$, as in \eqref{eq:infinite-payoff-extension-formula} and \eqref{eq:infinite-defect-extension}, and convergence in $\mathcal Y$ is weak-star convergence on total-variation-bounded sets. Thus, absolutely continuous pure rates may converge to measures with atoms, whose reconstructed states belong to $\mathcal X$.

However, there is a qualification necessitated by the unbounded competitor space. If the safe-stress condition is violated, competitors can be scaled without bound and the unrestricted diagonal defect is $+\infty$. Consequently, the approximating rates of Proposition~\ref{prop:linear-measure-attainment} need not have vanishing unrestricted defect. The natural relaxation is instead tested on every bounded competitor class. To this end, for $M>0$ we consider the restricted diagonal defects $\mathcal I_M(\mu)$ as in (\ref{eq:linear-bounded-defect}) with representation \eqref{eq:linear-bounded-defect-representation}. If the state-generating rate is pure, $\mu=v(t)\,dt$ with $v\in L^1(0,T;H)$, then $R\mu\in W^{1,1}(0,T;H)$ and, under the $W^{1,1}$ loading hypothesis of Section~\ref{subsec:linear-growth-energy-dissipation}, $g_\mu=S\widehat{R\mu}-h=SR\mu-h$ is continuous on $[0,T]$. For such $\mu$, the supremum in $\mathcal I_M(\mu)$ has the same value if the competitors are restricted to absolutely continuous measures with $L^1$ densities. Indeed, for every $\nu\in\mathcal Y$ with $|\nu|([0,T])\leq M$, there are $\nu_k=v_k(t)\,dt$ such that $\nu_k\stackrel{*}{\rightharpoonup}\nu$ and $\|v_k\|_{L^1}\leq|\nu|([0,T])\leq M$. Since $g_\mu$ is continuous, $\int_{[0,T]}(g_\mu,d\nu_k)_H\to\int_{[0,T]}(g_\mu,d\nu)_H$, and therefore $\limsup_{k\to\infty}\mathcal J(\nu_k;\mu)\leq\mathcal J(\nu;\mu)$. The infimum of $\mathcal J(\,\cdot\,;\mu)$ over bounded absolutely continuous competitors consequently agrees with the infimum over all bounded measure competitors. This conclusion is specific to pure state-generating rates: for a general $\mu\in\mathcal Y$, the centered field $g_\mu=S\widehat{R\mu}-h$ may be discontinuous at atoms, and weak-star approximation of competitors need not preserve the work term.

Proposition~\ref{prop:linear-measure-attainment} supplies a pure sequence $\mu_n=v_n(t)\,dt$ and a measure equilibrium $\mu^*\in\mathcal Y$ such that $\mu_n\stackrel{*}{\rightharpoonup}\mu^*$, $\mathcal J(\mu_n;\mu_n)\to0$, and $\mathcal Q(\mu_n)\to0$. It therefore follows that $\mathcal I_M(\mu_n)\to0$, for every fixed $M>0$. By the preceding equivalence, the same sequence has vanishing defect when the bounded competitors are restricted to pure rates. Hence $\mu^*$ belongs to the weak-star sequential closure of the pure almost-equilibria for every bounded competitor game. In the present $W^{1,1}$-loading setting, the compactness passage also yields $\mu^*=v^*(t)\,dt$, for some $v^*\in L^1(0,T;H)$. The relaxation is therefore lossless: the measure space supplies the compactness needed for existence, whereas the relaxed equilibrium ultimately lies again in the pure space. For rougher loading, as noted in Remark~\ref{ZhGtXu}, the same ambient relaxation may retain atoms or a singular continuous part.

To relate this construction precisely to Propositions~\ref{prop:abstract-relaxation-equivalence} and~\ref{prop:relaxed-equilibrium-compactness}, fix $R>L+L_0$ and set
\begin{equation}
Y_R:=\bigl\{v(t)\,dt:v\in L^1(0,T;H),\ \|v\|_{L^1}\leq R\bigr\},\quad \mathcal Y_R:=\bigl\{\mu\in\mathcal Y:|\mu|([0,T])\leq R\bigr\}.
\end{equation}
The weak-star topology on $\mathcal Y_R$ is metrizable and sequentially compact. For $\mu\in Y_R$, the pure bounded-competitor defect agrees, by the preceding argument, with the bounded measure-competitor defect and therefore is
\begin{equation}
I_R(\mu)=\mathcal J(\mu;\mu)+R\mathcal Q(\mu),\quad \mu\in Y_R.
\end{equation}
The sequence supplied by Proposition~\ref{prop:linear-measure-attainment} belongs eventually to $Y_R$ and satisfies $\mu_n\stackrel{*}{\rightharpoonup}\mu^*$ and $I_R(\mu_n)\to0$. Proposition~\ref{prop:relaxed-equilibrium-compactness} therefore shows that $\mu^*$ is a relaxed equilibrium of the $R$-bounded game, while Proposition~\ref{prop:abstract-relaxation-equivalence} gives the equivalent characterization $\overline{\mathcal I}_R(\mu^*)=0$. Thus, Proposition~\ref{prop:linear-measure-attainment} realizes concretely the compactness and recovery-sequence mechanism of Propositions~\ref{prop:abstract-relaxation-equivalence} and~\ref{prop:relaxed-equilibrium-compactness}.

} \hfill$\square$
\end{example}

\section{Example of exact relaxation: Non-convex kinetics with local memory}
\label{sec:local-memory-banding}

We now give an example in which the relaxation mechanism of Section~\ref{sec:direct-almost-equilibrium-relaxation} can be computed completely. The model is a quasistatic simple-shear problem with a non-convex local kinetic potential and a local internal variable carrying memory. The spatial model is deliberately unregularized: no interfacial energy, rate-gradient penalty, or diffusion of the memory variable is included. Consequently, the theory does not select a band width or propagation speed. Instead, it isolates the closure problem caused by non-convex kinetics when different microscopic rate populations acquire different internal histories, and it shows that the correct relaxed object is a Young measure on complete local histories rather than an instantaneous convexification of the rate potential.

The example is relevant for several reasons. First, it is genuinely evolutionary: the kinetic response at time $t$ depends on an internal variable generated by the preceding rate history. Second, the spatial recovery problem is elementary enough to be solved exactly, because the microscopic rate is not subject to a state-compatibility obstruction such as encountered, e.g., in multi-slip plasticity \cite{ContiOrtiz2005}. Third, the resulting enriched space gives a concrete instance of the abstract construction of Section~\ref{sec:direct-almost-equilibrium-relaxation}: the sequential lower-semicontinuous envelope of the pure diagonal defect is identified explicitly, every relaxed zero admits a pure vanishing-defect recovery sequence, and relaxed equilibria exist by a direct frozen-best-response construction.

The constitutive assumptions below are intentionally broad. They include smooth double-well kinetic potentials, non-monotone flow curves, and memory laws motivated by dynamic strain aging, thermal softening and damage. The Portevin--Le Ch\^atelier effect provides a particularly natural physical connection: negative strain-rate sensitivity, local aging variables, and propagating or intermittently localized deformation bands are central features of that phenomenon \cite{Yilmaz2011,McCormick1988,Ananthakrishna2007}. The present model is not intended as a detailed theory of PLC band propagation. Rather, it enables an exact relaxation theorem for the local-memory kinetic mechanism characteristic of such models.

\subsection{Simple-shear reduction and local-memory kinetics}
\label{subsec:memory-antiplane-reduction}

Let the transverse thickness be $H>0$ and set $\Omega=(0,H)$. The simple-shear displacement is $u:\Omega\times[0,T]\to\mathbb R$, the plastic shear is $\gamma:\Omega\times[0,T]\to\mathbb R$, and the prescribed end displacement is $\bar u\in W^{1,\infty}(0,T)$. We impose $u(0,t)=0$ and $u(H,t)=\bar u(t)$, and denote the corresponding macroscopic shear by ${\bar\xi}(t):=\bar u(t)/H$. For fixed $\gamma$ and $t$, the elastic energy is
\begin{equation}
\mathcal E(u,\gamma,t):=\frac{\mu_0}{2}\int_0^H |u_x(x)-\gamma(x)|^2\,dx,
\quad \mu_0>0.
\end{equation}
The displacement can be eliminated exactly.

\begin{lemma}[Simple-shear elastic reduction]
\label{lem:memory-elastic-reduction}
For every $\gamma\in L^2(0,H)$ and $t\in[0,T]$,
\begin{equation}
E(\gamma,t):=\inf_{\substack{u\in H^1(0,H)\\u(0)=0,\ u(H)=\bar u(t)}}\mathcal E(u,\gamma,t)
=\frac{\mu_0 H}{2}\left({\bar\xi}(t)-\bar\gamma\right)^2,
\quad
\bar\gamma:=\fint_0^H\gamma(x)\,dx.
\end{equation}
The minimizing elastic stress is spatially constant and is $\tau(\gamma,t)=\mu_0\left({\bar\xi}(t)-\bar\gamma\right)$. Moreover, for every $z\in L^2(0,H)$, 
\begin{equation}
D_1 E(\gamma,t)[z]=-H\tau(\gamma,t)\fint_0^H z(x)\,dx .
\end{equation}
\end{lemma}

\begin{proof}
The Euler equation for the displacement gives $\mu_0(u_x-\gamma)={\bar\tau}$, with ${\bar\tau}$ independent of $x$. Averaging over $(0,H)$ and using the boundary values yields ${\bar\xi}=\bar\gamma+{\bar\tau}/\mu_0$, hence the formula for ${\bar\tau}$. Substitution gives the reduced energy, and differentiation in $\gamma$ gives the last identity.
\end{proof}

Let $A=\mathbb R^m$ be the memory space and let ${L}=[-l_0,l_0]$ be a compact admissible rate interval, with $l_0>0$. The compact rate interval is introduced to keep the relaxation theorem transparent. A superlinear kinetic potential on $\mathbb R$ may be treated by truncation on defect sublevels, because bounded loading and superlinear growth bound all local minimizers on the relevant stress range.

The memory variable function $a:\Omega\times[0,T]\to A$ evolves according to the local kinetic relation $\dot a(x,t)=f\bigr(a(x,t),{l}(x,t)\bigl)$, where ${l}=\dot\gamma$ is the plastic shear rate. We assume throughout that $f:A\times {L}\to A$ is continuous, globally Lipschitz in $a$ uniformly in $r\in {L}$, and satisfies 
\begin{equation} \label{CU22TD}
|f(a,r)|\leq C_f(1+|a|),\quad \forall (a,r)\in A\times {L}. 
\end{equation}
The kinetic potential $\psi:{L}\times A\to[0,+\infty)$ is continuous. No convexity of the mapping $r\mapsto\psi(r;a)$ is required. We assumed homogeneous initial data $\gamma_0\in\mathbb R$ and $a_0\in A$ for simplicity. Spatially varying bounded initial data can be included by retaining the initial value as an additional coordinate of the history space below.

For ${l}\in L^\infty(\Omega\times(0,T))$ with ${l}(x,t)\in {L}$ a.e., define the causal reconstructions 
\begin{equation}
\gamma_{l}(x,t):=\gamma_0+\int_0^t {l}(x,s)\,ds ,
\quad
a_{l}(x,t):=a_0+\int_0^t f(a_{l}(x,s),{l}(x,s))\,ds .
\end{equation}
Their spatially averaged plastic shear and the corresponding stress are
\begin{equation}
\bar\gamma_l(t):=\fint_0^H\gamma_l(x,t)\,dx,
\quad
{\bar\tau}_{l}(t):=\mu_0\left({\bar\xi}(t)-\bar\gamma_l(t)\right),
\label{eq:memory-stress-reconstruction}
\end{equation}
respectively. The first reconstruction is Volterra, while the second is a nonlinear causal response. Gronwall's inequality gives a constant $C_a$, depending only on $T$, $a_0$, and $C_f$, such that $\|a_{l}\|_{L^\infty(\Omega\times(0,T))}\leq C_a$ for every admissible ${l}$. Likewise, $\gamma_{l}$ and ${\bar\tau}_{l}$ are uniformly bounded. Differentiating the reduced elastic energy along an absolutely continuous path gives
\begin{equation}
\frac{d}{dt}E(\gamma_{l}(t),t)
=H{\bar\tau}_{l}(t)\left(\dot{{\bar\xi}}(t)-\fint_0^H {l}(x,t)\,dx\right)
\end{equation}
for a.e.~$t$. Thus the stress work against the plastic rate is exactly the state derivative appearing in the energy--dissipation rate problem.

\subsection{The pure diagonal game and its defect}
\label{subsec:memory-pure-game}

For $|a|\leq C_a$ and for stresses in the bounded range generated by \eqref{eq:memory-stress-reconstruction}, define the local minimum value $m(a,{b}):=\min_{r\in {L}}\{\psi(r;a)-{b} r\}$ and the local excess by
\begin{equation}
 g(r;a,{b}):=\psi(r;a)-{b} r-m(a,{b})\geq0.
\label{eq:memory-local-gap}
\end{equation}
The minimizing set is then
\begin{equation}
{M}(a,{b}):=\mathop{\rm argmin}_{r\in {L}}\bigl\{\psi(r;a)-{b} r\bigr\}.
\label{eq:memory-best-response-set}
\end{equation}
Compactness of ${L}$ and continuity of $\psi$ imply that ${M}(a,{b})$ is nonempty and compact. The functions $m$ and $g$ are continuous on every bounded $(a,{b})$ set.

Let $Y:=\{l\in L^\infty(\Omega\times(0,T)):l(x,t)\in {L}\text{ a.e.}\}$. For $w,l\in Y$, define the normalized payoff
\begin{equation}
J(w;l):=\int_0^T\int_0^H g\bigl(w(x,t);a_{l}(x,t),{\bar\tau}_{l}(t)\bigr)\,dx\,dt.
\label{eq:memory-pure-payoff}
\end{equation}
The first argument is the competitor rate, while the second generates both the mechanical state and the local memory. Up to subtraction of the state-dependent minimum $m$, \eqref{eq:memory-pure-payoff} is the integral of $\psi(w;a_{l})+D_1 E(\gamma_{l},t)[w]$. Thus the normalization does not change the best-response relation.

\begin{proposition}[Pure defect and local best response]
\label{prop:memory-pure-defect}
For every $l\in Y$, the diagonal defect of Section~\ref{sec:direct-almost-equilibrium-relaxation} is
\begin{equation}
I(l)=J(l;l)=\int_0^T\int_0^H g\bigl(l(x,t);a_{l}(x,t),{\bar\tau}_{l}(t)\bigr)\,dx\,dt.
\label{eq:memory-pure-defect}
\end{equation}
Consequently, $l$ is a pure diagonal equilibrium if and only if $l(x,t)\in{M}(a_{l}(x,t),{\bar\tau}_{l}(t))$ for a.e.~$(x,t)\in\Omega\times(0,T)$.
\end{proposition}

\begin{proof}
For fixed $l$, the local gap is nonnegative, so $J(w;l)\geq0$ for every competitor. The multifunction $(x,t)\mapsto{M}(a_{l}(x,t),{\bar\tau}_{l}(t))$ has nonempty compact values and a measurable graph. A measurable selection $w^*$ therefore exists \cite[Thm.~18.19]{AliprantisBorder2006}, with $w^*(x,t)\in{M}(a_{l}(x,t),{\bar\tau}_{l}(t))$ a.e. Hence $w^*\in Y$ and $J(w^*;l)=0$. It follows from the definition of the diagonal defect that $I(l)=J(l;l)$. Since the integrand in \eqref{eq:memory-pure-defect} is nonnegative, $I(l)=0$ is equivalent to the asserted local best-response condition.
\end{proof}

\begin{remark}[Why pointwise convexification is not enough]
\label{rem:memory-static-convexification}
{\rm If the memory variable were absent, relaxation of the local non-convex rate problem would be described at a fixed state by the lower convex envelope of $r\mapsto\psi(r)-{\bar\tau} r$. With memory, replacing a mixture of rates by its barycenter generally destroys future information. Suppose, for example, that at time $t_0$ a population with common memory $a_0$ splits into rates $r_-$ and $r_+$ with fractions $\theta$ and $1-\theta$. Then, over a short interval $h$, $a_\pm(t_0+h)=a_0+h f(a_0,r_\pm)+o(h)$. The two populations therefore acquire different memories whenever $f(a_0,r_-)\neq f(a_0,r_+)$. Replacing them by the average rate $\bar r=\theta r_-+(1-\theta)r_+$ gives instead the single memory increment $h f(a_0,\bar r)$, which need not equal the mixture of the two increments. Even when the first increments agree serendipitously, the future kinetic potential depends on the distribution of the memories, not merely on their mean. The relaxation must therefore retain temporal coupling information.} \hfill$\square$
\end{remark}

\subsection{Relaxed local histories and the enriched space}
\label{subsec:memory-history-space}

The natural compactification of one local rate history is the relaxed-control compactification, see e.g.~\cite{Warga1972}. A relaxed control is a positive Radon measure $\lambda$ on $[0,T]\times {L}$ whose first marginal is Lebesgue measure. We write its disintegration as $d\lambda(t,r)=d\lambda_t(r)\,dt$, with $\lambda_t\in\mathcal P({L})$ for a.e.~$t \in (0,T)$. Here and subsequently, $\mathcal{P}(L)$ denotes the space of Borel probability measures on $L$. A relaxed local history is a triple $\omega=(\gamma,a,\lambda)$ satisfying
\begin{equation}
\gamma(t)=\gamma_0+\int_0^t\int_{L} r\,d\lambda_s(r)\,ds,
\label{eq:memory-relaxed-gamma}
\end{equation}
\begin{equation}
a(t)=a_0+\int_0^t\int_{L} f(a(s),r)\,d\lambda_s(r)\,ds.
\label{eq:memory-relaxed-a}
\end{equation}
Let $\mathcal X$ denote the set of all such histories. We equip $\mathcal X$ with uniform convergence of $\gamma$ and $a$, together with narrow convergence of $\lambda$ as a measure on the compact set $[0,T]\times {L}$. A history is called pure if $\lambda_t=\delta_{r(t)}$ for a.e.~$t$ for some measurable $r:[0,T]\to {L}$.

\begin{lemma}[Compactness and chattering of local histories]
\label{lem:memory-local-history-compactness}
The space $\mathcal X$ is compact and metrizable. Pure histories are dense in $\mathcal X$. More precisely, for every $\omega=(\gamma,a,\lambda)\in\mathcal X$ there exist pure histories $\omega_h=(\gamma_h,a_h,\delta_{r_h(t)})$ such that $\gamma_h\to\gamma$ and $a_h\to a$ uniformly on $[0,T]$, and $\delta_{r_h(t)}\,dt\rightharpoonup\lambda$ narrowly in $\mathcal M([0,T]\times {L})$. If $c:[0,T]\times A\times {L}\to\mathbb R$ is continuous on the bounded state range generated above, the sequence may be chosen so that
\begin{equation}
\int_0^T c(t,a_h(t),r_h(t))\,dt
\to
\int_0^T\int_{L} c(t,a(t),r)\,d\lambda_t(r)\,dt.
\label{eq:memory-local-chattering-cost}
\end{equation}
\end{lemma}

\begin{proof}
Because ${L}$ is compact, the set of relaxed controls is narrowly compact and the corresponding $\gamma$ paths are uniformly Lipschitz, see e.g.~\cite{AmbrosioGigliSavare2008}. Gronwall's inequality and the growth assumption (\ref{CU22TD}) on $f$ give a common bound on $a$, after which the continuity of $f$ gives a common Lipschitz bound on the $a$ paths. Arzel\`a--Ascoli and narrow compactness yield a convergent subsequence. Passing to the limit in \eqref{eq:memory-relaxed-gamma}--\eqref{eq:memory-relaxed-a} is justified by uniform convergence of $a$, continuity of $f$, and narrow convergence of the controls, so $\mathcal X$ is closed and therefore compact.

For density, first approximate $\lambda$ on a finite time partition by controls that are constant as probability measures on each interval, and approximate each of those probability measures by a finite atomic measure on ${L}$. On every time cell, divide the cell into subintervals having the corresponding atomic proportions and let $r_h$ take the associated atomic values. The resulting Dirac controls converge narrowly to $\lambda$. Continuous dependence of the integral equations on the relaxed control gives uniform convergence of $\gamma_h$ and $a_h$. Applying the same construction to the finite family of continuous functions needed to approximate $c$ uniformly gives \eqref{eq:memory-local-chattering-cost}. This is the usual chattering construction, see e.g.~\cite{Warga1972}, written here at the level needed for the defect.
\end{proof}

To retain spatial microstructure and band-specific histories, define the enriched rate-path space
\begin{equation}
\mathcal Y:=\bigl\{\Lambda\in\mathcal M([0,H]\times\mathcal X):\Lambda\geq0,\ (\pi_1)_\#\Lambda=\mathcal L^1\mres[0,H]\bigr\},
\label{eq:memory-enriched-space}
\end{equation}
where $\pi_1$ is projection onto the spatial coordinate. Every measure $\Lambda\in\mathcal Y$ has total mass $H$. Since $[0,H]\times\mathcal X$ is compact, $\mathcal Y$ is compact and metrizable for narrow convergence.

Every pure rate $l\in Y$ generates a pure local history $\omega_{l}(x)=(\gamma_{l}(x,\cdot),a_{l}(x,\cdot),\delta_{l(x,t)})$ for a.e.~$x$. We identify $l$ with its graph history measure
\begin{equation}
\Lambda_{l}:=(x\mapsto(x,\omega_{l}(x)))_\#\mathcal L^1\mres[0,H]\in\mathcal Y.
\label{eq:memory-pure-graph-measure}
\end{equation}
The spatial coordinate makes this embedding injective modulo equality a.e. A general measure $\Lambda\in\mathcal Y$ disintegrates as $d\Lambda(x,\omega) = d\Lambda_x(\omega)\,dx$. Non-Dirac measures $\Lambda_x$ record microscopic populations of complete histories near the macroscopic point $x$.

For $\Lambda\in\mathcal Y$, define its averaged plastic shear and reconstructed stress as
\begin{equation}
\bar\gamma_\Lambda(t):=\frac{1}{H}\int_{[0,H]\times\mathcal X}\gamma(t)\,d\Lambda(x,\omega),
\quad
{\bar\tau}_\Lambda(t):=\mu_0\bigl({\bar\xi}(t)-\bar\gamma_\Lambda(t)\bigr).
\label{eq:memory-relaxed-stress}
\end{equation}
The candidate relaxed defect is
\begin{equation}
\mathscr I(\Lambda)
:=
\int_{[0,H]\times\mathcal X}
\int_0^T\int_{L}
 g\bigl(r;a(t),{\bar\tau}_\Lambda(t)\bigr)
\,d\lambda_t(r)\,dt\,d\Lambda(x,\omega).
\label{eq:memory-explicit-relaxed-defect}
\end{equation}
It is immediate that $\mathscr I\geq0$.

\begin{lemma}[Continuity of the explicit relaxed defect]
\label{lem:memory-relaxed-defect-continuity}
If $\Lambda_h\rightharpoonup\Lambda$ narrowly in $\mathcal Y$, then ${\bar\tau}_{\Lambda_h}\to{\bar\tau}_\Lambda$ uniformly on $[0,T]$ and $\mathscr I(\Lambda_h)\to\mathscr I(\Lambda)$.
\end{lemma}

\begin{proof}
For every fixed $t$, the evaluation $\omega\mapsto\gamma(t)$ is continuous on $\mathcal X$, so narrow convergence gives $\bar\gamma_{\Lambda_h}(t)\to\bar\gamma_\Lambda(t)$. All averaged shear paths are $l_0$-Lipschitz. Pointwise convergence together with this common modulus gives uniform convergence, and therefore uniform convergence of the stresses.

For a continuous stress path ${\bar\tau}$, define 
\begin{equation}
F_{\bar\tau}(\omega):=\int_0^T\int_{L} g(r;a(t),{\bar\tau}(t))\,d\lambda_t(r)\,dt .
\end{equation}
The topology of $\mathcal X$, continuity of $g$, and compactness of the common state range imply that $F_{\bar\tau}$ is continuous on $\mathcal X$. Moreover, $F_{{\bar\tau}_h}\to F_{\bar\tau}$ uniformly on $\mathcal X$ whenever ${\bar\tau}_h\to{\bar\tau}$ uniformly. Hence,
\begin{equation}
\mathscr I(\Lambda_h)=\int F_{{\bar\tau}_{\Lambda_h}}\,d\Lambda_h
\to
\int F_{{\bar\tau}_\Lambda}\,d\Lambda
=\mathscr I(\Lambda) ,
\end{equation}
as claimed.
\end{proof}

For pure graph measures the explicit relaxed defect agrees exactly with the original diagonal defect.

\begin{lemma}[Consistency on pure histories]
\label{lem:memory-pure-relaxed-consistency}
For every $l\in Y$, $\mathscr I(\Lambda_{l})=I(l)$.
\end{lemma}

\begin{proof}
For $\Lambda_{l}$, the local control is $\lambda_t=\delta_{l(x,t)}$, the averaged shear in \eqref{eq:memory-relaxed-stress} is $\fint_0^H\gamma_{l}(x,t)\,dx$, and therefore ${\bar\tau}_{\Lambda_{l}}={\bar\tau}_{l}$. Substitution into \eqref{eq:memory-explicit-relaxed-defect} gives \eqref{eq:memory-pure-defect}.
\end{proof}

\subsection{Exact relaxation of the diagonal defect}
\label{subsec:memory-exact-relaxation}

We can now identify the abstract relaxed defect of Section~\ref{sec:direct-almost-equilibrium-relaxation}. We use the narrow metric on $\mathcal Y$ and identify the pure rate $l$ with the graph measure $\Lambda_{l}$. The infinite extension of the pure defect is therefore the function $\mathcal I$ of \eqref{eq:infinite-defect-extension}, with value $I(l)$ on pure graph measures and $+\infty$ otherwise.

\begin{theorem}[Exact history-space relaxation]
\label{thm:memory-exact-relaxation}
Under the assumptions of Sections~\ref{subsec:memory-antiplane-reduction}--\ref{subsec:memory-history-space}, the sequential lower-semicontinuous envelope of the pure diagonal defect is exactly \eqref{eq:memory-explicit-relaxed-defect}:
\begin{equation}
\overline{\mathcal I}(\Lambda)=\mathscr I(\Lambda),
\quad \forall\Lambda\in\mathcal Y.
\label{eq:memory-relaxation-formula}
\end{equation}
More precisely, if $l_h\in Y$ and $\Lambda_{l_h}\rightharpoonup\Lambda$, then
\begin{equation}
\mathscr I(\Lambda)\leq\liminf_{h\to\infty}I(l_h).
\label{eq:memory-relaxation-liminf}
\end{equation}
Conversely, for every $\Lambda\in\mathcal Y$ there is a sequence $l_h\in Y$, which may be chosen piecewise-constant in space and time, such that
\begin{equation}
\Lambda_{l_h}\rightharpoonup\Lambda,
\quad
I(l_h)\to\mathscr I(\Lambda).
\label{eq:memory-relaxation-recovery}
\end{equation}
\end{theorem}

\begin{proof}
The liminf inequality follows immediately from Lemmas~\ref{lem:memory-relaxed-defect-continuity} and~\ref{lem:memory-pure-relaxed-consistency}; in fact, along every convergent pure sequence one has $I(l_h)=\mathscr I(\Lambda_{l_h})\to\mathscr I(\Lambda)$.

It remains to prove density of pure graph histories. Fix $\Lambda\in\mathcal Y$. Since $[0,H]\times\mathcal X$ is compact and the first marginal of $\Lambda$ is Lebesgue measure, a standard spatial chattering construction, see e.g.~\cite{Warga1972, CastaingRaynaudDeFitteValadier2004}, gives a sequence of graph measures $\widetilde\Lambda_h=(x\mapsto(x,\widetilde\omega_h(x)))_\#\mathcal L^1\mres[0,H]$ with $\widetilde\omega_h$ piecewise constant in $x$ and $\widetilde\Lambda_h \rightharpoonup \Lambda$. One way to construct the sequence is to partition $[0,H]$ into cells of diameter tending to zero, approximate each conditional measure $\Lambda_x$ on every cell by a finite atomic measure on $\mathcal X$, and subdivide the cell according to the atomic masses.

More explicitly, on each spatial cell $Q$ we first replace the family $x \mapsto \Lambda_x$ by its average
\begin{equation}
    \bar\Lambda_Q:=\fint_Q \Lambda_x\,dx,
\end{equation}
approximate $\bar\Lambda_Q$ narrowly by a finite atomic probability measure $\sum_j\theta_{Q,j}\delta_{\omega_{Q,j}}$, and partition $Q$ into measurable subsets $Q_j$ with $|Q_j|=\theta_{Q,j}|Q|$, setting $\widetilde\omega_h(x)=\omega_{Q,j}$ on $Q_j$.  Since the cell diameters tend to zero and continuous functions on the compact space $[0,H]\times\mathcal X$ are uniformly continuous, testing against any $\varphi\in C([0,H]\times\mathcal X)$ shows that the resulting graph measures $\widetilde\Lambda_h$ converge narrowly to $\Lambda$.

The finitely many histories used in each $\widetilde\omega_h$ may be relaxed rather than pure. Apply Lemma~\ref{lem:memory-local-history-compactness} to each of these finitely many histories, choosing the temporal chattering sufficiently fine that the resulting pure histories change both the graph measure and the value of $\mathscr I$ by at most $1/h$. Assign the corresponding pure rate histories to the same spatial subcells. This produces $l_h\in Y$, piecewise-constant in space and time, with $\Lambda_{l_h}\rightharpoonup\Lambda$. Lemma~\ref{lem:memory-relaxed-defect-continuity} and Lemma~\ref{lem:memory-pure-relaxed-consistency} then give
\begin{equation}
I(l_h)=\mathscr I(\Lambda_{l_h})\to\mathscr I(\Lambda).
\end{equation}
The recovery sequence proves the upper bound in the definition \eqref{eq:relaxed-diagonal-defect}, and \eqref{eq:memory-relaxation-formula} follows.
\end{proof}

The theorem gives an intrinsic characterization of the relaxed zero set.

\begin{corollary}[Relaxed local equilibrium]
\label{cor:memory-relaxed-zero}
A history Young measure $\Lambda\in\mathcal Y$ is a relaxed diagonal equilibrium if and only if
\begin{equation}
\lambda_{\omega,t}\bigl({M}(a_\omega(t),{\bar\tau}_\Lambda(t))\bigr)=1
\label{eq:memory-relaxed-support-condition}
\end{equation}
for $\Lambda$-a.e.~$(x,\omega)$ and a.e.~$t\in(0,T)$. Every such $\Lambda$ admits a pure recovery sequence $l_h\in Y$ satisfying $\Lambda_{l_h}\rightharpoonup\Lambda$ and $I(l_h)\to0$.
\end{corollary}

\begin{proof}
By \eqref{eq:memory-explicit-relaxed-defect}, $\mathscr I(\Lambda)$ is the integral of the nonnegative function $g$. Its integral vanishes exactly when $g=0$ for the product of $\Lambda$, Lebesgue measure in time, and the relaxed control. By \eqref{eq:memory-local-gap}--\eqref{eq:memory-best-response-set}, this is equivalent to \eqref{eq:memory-relaxed-support-condition}. Theorem~\ref{thm:memory-exact-relaxation} and Proposition~\ref{prop:abstract-relaxation-equivalence} identify this zero set with the relaxed equilibria of Section~\ref{sec:direct-almost-equilibrium-relaxation}, and the recovery statement is \eqref{eq:memory-relaxation-recovery} at zero defect.
\end{proof}

\begin{remark}[The payoff is not extended to relaxed histories]
\label{rem:memory-no-relaxed-payoff}
{\rm Theorem~\ref{thm:memory-exact-relaxation} computes the relaxed defect, not a new pointwise game on $\mathcal Y$. At a non-Dirac history measure there is no canonical way to pair an arbitrary relaxed competitor with the microscopic state histories retained by $\Lambda$. This is exactly the distinction emphasized in Remark~\ref{JeyRfv} and in the discussion following \eqref{eq:infinite-payoff-extension-formula}: pure competitors define the approximating games, whereas the enriched object is characterized by the closure of their diagonal defects.} \hfill$\square$
\end{remark}

\subsection{Existence from frozen best responses}
\label{subsec:memory-existence}

In the present example, the vanishing-defect approximability required abstractly in Section~\ref{sec:direct-almost-equilibrium-relaxation} can be proved constructively. The argument also gives a simple time-stepping approximation.

\begin{theorem}[Existence and constructive almost equilibria]
\label{thm:memory-existence}
There exists at least one relaxed diagonal equilibrium $\Lambda^*\in\mathcal Y$. More precisely, for every sequence of partitions
\begin{equation}
0=t_0^h<t_1^h<\cdots<t_{N_h}^h=T,
\quad
\Delta_h:=\max_k(t_{k+1}^h-t_k^h)\downarrow0,
\end{equation}
one can construct $l_h\in Y$, piecewise constant in time, such that on every interval $[t_k^h,t_{k+1}^h)$,
\begin{equation}
l_h(x,t)\in M\bigl(a_{l_h}(x,t_k^h),{\bar\tau}_{l_h}(t_k^h)\bigr) ,
\quad\text{for a.e. }x \in (0,H).
\label{eq:memory-frozen-best-response}
\end{equation}
The resulting defects satisfy
\begin{equation}
I(l_h)\to0.
\label{eq:memory-frozen-defect-zero}
\end{equation}
Every subsequential narrow limit of $\Lambda_{l_h}$ is a relaxed diagonal equilibrium.
\end{theorem}

\begin{proof}
Assume that the path has been constructed up to $t_k^h$. Continuity of $\psi$ implies that the multifunction $x\mapsto{M}(a_{l_h}(x,t_k^h),{\bar\tau}_{l_h}(t_k^h))$ has a measurable graph and nonempty compact values. Choose a measurable selection $r_k^h(x)$ and set $l_h(x,t)=r_k^h(x)$ on $[t_k^h,t_{k+1}^h)$. The reconstruction equations then determine $\gamma_{l_h}$ and $a_{l_h}$ on that interval. Iterating gives the full path and \eqref{eq:memory-frozen-best-response}. If a minimizing set contains several separated rates, the selection may be chosen differently on different spatial subsets, thereby creating band populations without changing the estimate below.

All rates satisfy $|l_h|\leq l_0$. The memory estimate and the evolution equation therefore give a constant $C_1$ independent of $h$ such that 
\begin{equation}
|a_{l_h}(x,t)-a_{l_h}(x,t_k^h)|\leq C_1\Delta_h, \quad t\in[t_k^h,t_{k+1}^h) .
\end{equation}
From \eqref{eq:memory-stress-reconstruction}, 
\begin{equation}
\dot{\bar\tau}_{l_h}(t)=\mu_0(\dot{{\bar\xi}}(t)-\fint_0^H l_h(x,t)\,dx),
\quad \text{for a.e.}\ t\in(0,T) , 
\end{equation}
hence 
\begin{equation}
|{\bar\tau}_{l_h}(t)-{\bar\tau}_{l_h}(t_k^h)|\leq C_2\Delta_h .
\end{equation}
Let $\omega_g$ be a common modulus of continuity of $g$ on the compact set containing ${L}$, all reconstructed memories, and all reconstructed stresses. By \eqref{eq:memory-frozen-best-response}, 
\begin{equation}
g(l_h(x,t);a_{l_h}(x,t_k^h),{\bar\tau}_{l_h}(t_k^h))=0 . 
\end{equation}
Therefore,
\begin{equation}
g\bigl(l_h(x,t);a_{l_h}(x,t),{\bar\tau}_{l_h}(t)\bigr)
\leq\omega_g\bigl((C_1+C_2)\Delta_h\bigr)
\end{equation}
for a.e.~$(x,t)$. Integration gives
\begin{equation}
0\leq I(l_h)\leq H T\,\omega_g\bigl((C_1+C_2)\Delta_h\bigr)\to0.
\end{equation}
The space $\mathcal Y$ is compact, so $\Lambda_{l_h}$ has a narrowly convergent subsequence. Proposition~\ref{prop:relaxed-equilibrium-compactness}, or equivalently Theorem~\ref{thm:memory-exact-relaxation}, shows that every subsequential limit is a relaxed diagonal equilibrium.
\end{proof}

\begin{remark}[What the construction proves]
\label{rem:memory-existence-strength}
{\rm The theorem verifies the approximability hypothesis that is substantive in the abstract direct method. Compactness is built into the history space, and Theorem~\ref{thm:memory-exact-relaxation} supplies the recovery and lower-semicontinuity statements. Thus all three ingredients of the abstract relaxation method are explicit here: vanishing-defect approximation, compactness, and identification of the relaxed defect. } \hfill$\square$
\end{remark}

\begin{example}[Serrated recovery and band-specific memory]
{\rm Under general loading, different rate populations acquire different memories and are represented more naturally by distinct complete histories. To illustrate this point, let $\Lambda\in\mathcal Y$ be a relaxed zero satisfying \eqref{eq:memory-relaxed-support-condition}, and suppose for purposes of illustration that its history marginal consists of two populations,
\begin{equation}
d\Lambda(x,\omega)
=dx\bigl(\theta\,\delta_{\omega^-}(d\omega)
+(1-\theta)\delta_{\omega^+}(d\omega)\bigr),
\quad 0<\theta<1,
\label{eq:serrated-two-history-measure}
\end{equation}
with $\omega^\pm=(\gamma^\pm,a^\pm,\lambda^\pm)$. Thus, for a.e.~$t$,
$\lambda_t^\pm$ is supported on
$M(a^\pm(t),\bar\tau_\Lambda(t))$, but in general
$a^-(t)\neq a^+(t)$. The relaxed macroscopic plastic rate is
\begin{equation}
\bar r_\Lambda(t)
:=\theta\int_L r\,d\lambda_t^-(r)
+\bigl(1-\theta\bigr)\int_L r\,d\lambda_t^+(r),
\label{eq:serrated-relaxed-rate}
\end{equation}
and differentiation of \eqref{eq:memory-relaxed-stress} gives
\begin{equation}
\dot{\bar\tau}_\Lambda(t)
=\mu_0\bigl(\dot{\bar\xi}(t)-\bar r_\Lambda(t)\bigr).
\label{eq:serrated-relaxed-stress-rate}
\end{equation}
History dependence is already evident for a simple aging law of the form
\begin{equation}
f(a,r)=(1-a)/t_a-\kappa|r|a . 
\end{equation}
Indeed,
\begin{equation}
\dot a^\pm(t)
=\frac{1-a^\pm(t)}{t_a}
-\kappa a^\pm(t)\int_L |r|\,d\lambda_t^\pm(r).
\label{eq:serrated-memory-mixture}
\end{equation}
If $\rho_\pm:=\int_L|r|\,d\lambda_t^\pm(r)$ are constant on an interval $[t_0,t_1]$, then
\begin{equation}
a^\pm(t)
=a_{\rm ss}^\pm+\bigl(a^\pm(t_0)-a_{\rm ss}^\pm\bigr)
\exp\left[-\beta_\pm(t-t_0)\right],
\quad
a_{\rm ss}^\pm:=\frac{1}{1+\kappa t_a\rho_\pm},
\quad
\beta_\pm:=\frac{1}{t_a}+\kappa\rho_\pm.
\label{eq:serrated-memory-explicit}
\end{equation}
Hence $\rho_-\neq\rho_+$ produces distinct memory paths even if the two
populations start from the same value. The limit therefore retains more than
the barycentric rate: it retains the correlation between each rate history and
the memory generated by that history.

Corollary~\ref{cor:memory-relaxed-zero} and Theorem~\ref{thm:memory-exact-relaxation} provide pure, piecewise-constant recovery sequences $l_h$ such that $\Lambda_{l_h}\rightharpoonup\Lambda$, and $I(l_h)\to0$. The temporal chattering in the proof of Lemma~\ref{lem:memory-local-history-compactness} may be ordered on every small time cell $I_k^h=[t_k^h,t_{k+1}^h)$ so that the lower-rate selections occur first and the higher-rate selections second, without changing their cell fractions. Writing
\begin{equation}
\bar{q}_h(t):=\fint_0^H l_h(x,t)\,dx-\bar r_\Lambda(t),
\label{eq:serrated-rate-fluctuation}
\end{equation}
the construction gives a uniformly bounded $\bar{q}_h$ whose primitive is small:
there is $C>0$, independent of $h$, such that
\begin{equation}
\sup_{0\leq t\leq T}\left|\int_0^t \bar{q}_h(s)\,ds\right|
\leq C\Delta_h,
\quad
\Delta_h:=\max_k|I_k^h|\downarrow0.
\label{eq:serrated-primitive-estimate}
\end{equation}
The pure stress satisfies
\begin{equation}
\dot{\bar\tau}_{l_h}=\mu_0\bigr(\dot{\bar\xi}-\fint_0^H l_h\,dx\bigl) , 
\end{equation}
and therefore
\begin{equation}
\bar\tau_{l_h}(t)-\bar\tau_\Lambda(t)
=-\mu_0\int_0^t \bar{q}_h(s)\,ds,
\quad
\|\bar\tau_{l_h}-\bar\tau_\Lambda\|_{L^\infty(0,T)}
\leq\mu_0 C\Delta_h\to0.
\label{eq:serrated-uniform-stress}
\end{equation}
Thus the pure approximants may exhibit alternating stress rises and drops on
the chattering scale while converging uniformly to a smooth relaxed stress
path. The serrations disappear from the macroscopic limit, but the limiting
history Young measure in \eqref{eq:serrated-two-history-measure} remains
non-Dirac and retains the band-specific memories.

\begin{figure}[t]
    \centering
    \begin{subfigure}[t]{0.48\textwidth}
        \centering
        \includegraphics[width=\textwidth]{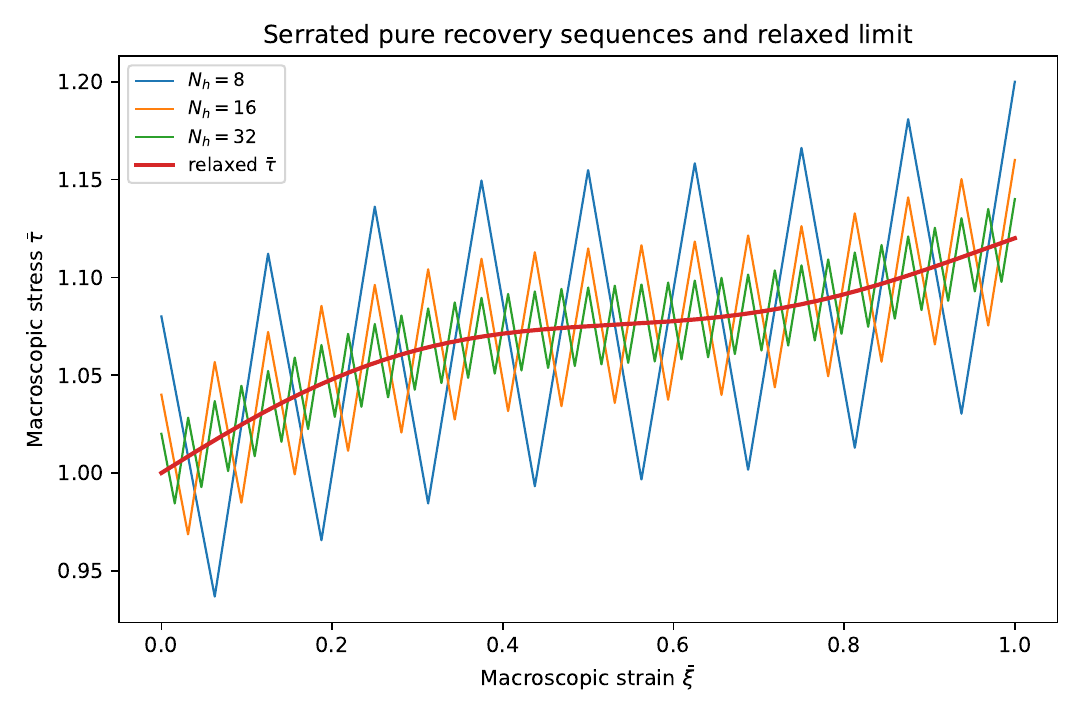}
        \caption{Pure serrated stress--strain paths for
        $N_h=8,16,32$ and the relaxed limit.}
    \end{subfigure}
    \hfill
    \begin{subfigure}[t]{0.48\textwidth}
        \centering
        \includegraphics[width=\textwidth]{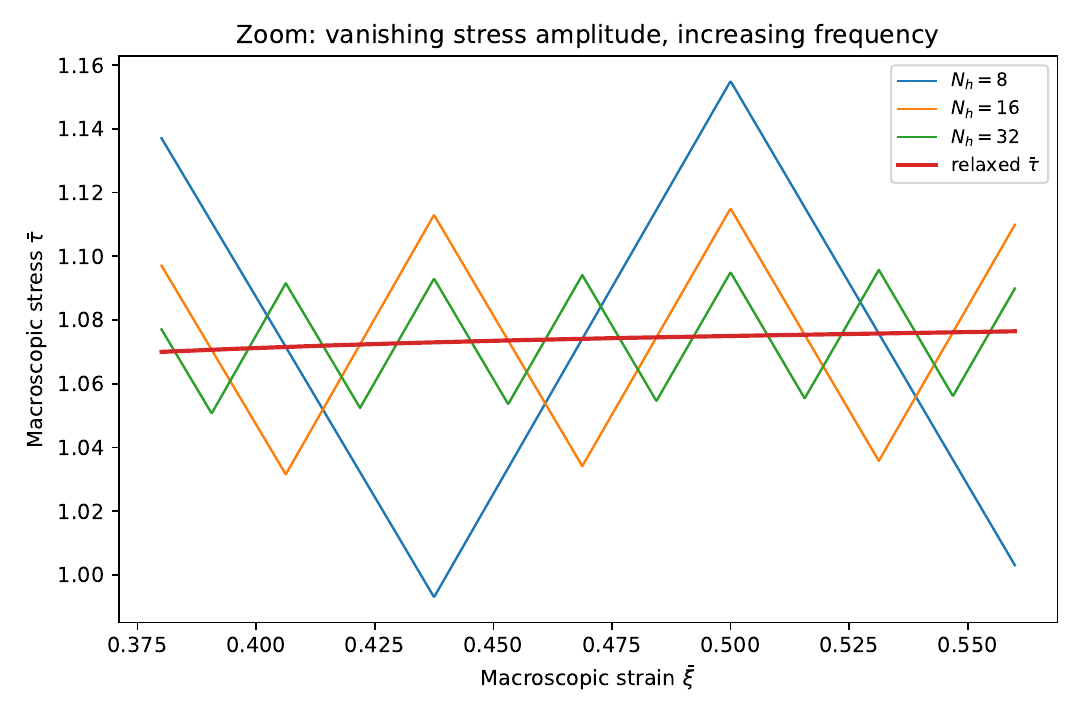}
        \caption{Zoom showing vanishing serration amplitude and increasing
        switching frequency.}
    \end{subfigure}
    \caption{Serrated pure almost-equilibrium recovery sequences converging
    uniformly to a relaxed macroscopic stress--strain path. Although the
    macroscopic oscillations vanish, $\Lambda_{l_h}\rightharpoonup\Lambda$
    retains the distinct complete histories and their rate--memory
    correlations.}
    \label{fig:serrated-recovery}
\end{figure}

Figure~\ref{fig:serrated-recovery} illustrates this mechanism numerically. For monotone loading $\bar\xi=t$, take
\begin{equation}
\bar\tau(\bar\xi)
=1+0.18\bar\xi-0.06\bar\xi^2+0.015\sin(2\pi\bar\xi),
\quad 0\leq\bar\xi\leq1,
\label{eq:serrated-plot-limit}
\end{equation}
and define
\begin{equation}
\bar\tau_h(\bar\xi)
=\bar\tau(\bar\xi)
+\frac{\eta}{N_h}\operatorname{tri}(N_h\bar\xi),
\quad
\eta=0.64,
\quad
N_h\in\{8,16,32\},
\label{eq:serrated-plot-sequence}
\end{equation}
where $\operatorname{tri}(s):=4|s-\lfloor s\rfloor-1/2|-1$. Then $\|\bar\tau_h - \bar\tau\|_{L^\infty}\leq\eta/N_h\to0$, whereas away from corners the associated average plastic rate satisfies
\begin{equation}
\bar r_h(\bar\xi)
=1-\frac{1}{\mu_0}\frac{d\bar\tau_h}{d\bar\xi}
=\bar r(\bar\xi)-\frac{4\eta}{\mu_0}\sigma_h(\bar\xi),
\quad
\sigma_h(\bar\xi)\in\{-1,1\}.
\label{eq:serrated-plot-rate-switching}
\end{equation}
Thus, the rate switching remains of fixed amplitude while its period tends to zero, whereas the induced stress serrations have amplitude $O(N_h^{-1})$. The plotted curves are a schematic visualization of the recovery kinematics; they do not presuppose an additional special form of a kinetic potential density $\psi$ beyond the support condition \eqref{eq:memory-relaxed-support-condition}.}\hfill$\square$
\end{example}

\section{Summary and concluding remarks}

We have formulated energy--dissipation evolution as a game between state and rate histories in which the state--rate consistency constraint is eliminated by the causal Volterra reconstruction $Rv=u_0+Kv$. The resulting evolution problem is a diagonal equilibrium for the reduced payoff $J(w;v)=G(w;Rv)$. This formulation separates the two occurrences of the rate variable and identifies the approximability, compactness, and complementary semicontinuity properties needed for existence. It is closely related to variational inequalities, equilibrium problems, monotone-operator theory, and convex integral functionals \cite{Stackelberg1934, BasarOlsder1999, Fan1961, BlumOettli1994, BrowderHess1972, LerayLions1965, Brezis1968, Zeidler1990, Showalter1997, Rockafellar1971IntegralII, Rockafellar1971ConvexIntegralDuality}, but acts on a causal diagonal rather than on a single static functional.

The examples developed here illustrate two different compactness regimes. In the representative superlinear-growth system of Section~\ref{subsec:superlinear-energy-dissipation}, the rate space is reflexive and the complete diagonal defect yields existence, uniqueness, and strong convergence of a temporal Galerkin approximation. In the representative linear-growth system of Section~\ref{subsec:linear-growth-energy-dissipation}, the natural compact ambient space for rates is the space of vector-valued Radon measures, with right-continuous $BV$ reconstructed states \cite{DiestelUhl1977,Reshetnyak1968,AmbrosioDalMaso1990,AFP00}. The centered state representative gives the exact quadratic chain rule at atoms, energy balance supplies the total-variation bound needed for weak-star compactness, and viscous regularization together with temporal measure approximation yields the unique equilibrium. Under the $W^{1,1}$ loading hypothesis used here, the limiting equilibrium rate is again absolutely continuous.

When compactness or lower semicontinuity fails in the pure path space, the relaxation developed in Section~\ref{sec:direct-almost-equilibrium-relaxation} closes pure almost equilibria whose diagonal defects vanish in an enriched topology. The relaxed object is therefore the sequential closure of the evolutionary equilibrium principle, rather than a pointwise extension of the original game to generalized states. Section~\ref{sec:local-memory-banding} gives a complete realization of that construction. There the enriched state is a Young measure on complete local histories, and the sequential lower-semicontinuous envelope of the pure defect is identified explicitly with a nonnegative history-space gap functional. Its zero set is exactly the set of relaxed local best responses, every relaxed zero has a pure recovery sequence, and a frozen-best-response time discretization supplies the vanishing-defect approximability needed for existence.

The local-memory example also clarifies why a purely instantaneous relaxation can be insufficient. If different rate populations follow different kinetic branches, an internal variable driven by the local rate separates their subsequent histories. Collapsing those populations to a single average rate therefore loses information needed to predict future response. The history Young measure retains precisely this correlation. In a two-rate regime the common-tangent construction gives the usual coexistence picture, while the memory variable allows the two bands to evolve different local kinetic landscapes. This provides a reduced connection with dynamic strain aging, negative strain-rate sensitivity, and Portevin--Le Ch\^atelier-type localization \cite{Yilmaz2011,McCormick1988,Ananthakrishna2007}.

The example is intentionally unregularized in space. It therefore does not select a band width, a band location, or a propagation speed. Adding spatial diffusion of the internal variable, heat conduction, gradient regularization, or nonlocal dislocation transport would couple neighboring histories and turn the elementary chattering recovery into a genuine space--time recovery problem. Multi-slip crystal plasticity with latent hardening provides another natural next step: there the microscopic state is subject to differential compatibility and branch constraints, so recovery must preserve both kinetic action and spatial compatibility. These extensions are substantially harder and are not required for the exact local-memory relaxation proved here.

More generally, the Volterra reconstruction treated here is a particularly reducible causal constraint: its unique state response can be written explicitly and the two-player structure collapses to the diagonal bifunction $J(w;v)$. For other causal constraints the game structure should be more visible. Mass transport provides a natural example: the continuity equation couples density and flux through a PDE constraint, while the Benamou--Brenier action supplies a trajectory-level kinetic functional \cite{BenamouBrenier2000,AmbrosioGigliSavare2008}. Even when the state is uniquely determined in a suitable weak sense, eliminating it need not yield an elementary reconstruction operator or the most useful analytical formulation. It is therefore natural to study the unreduced state--rate equilibrium directly, together with compactness and relaxation of both responses. Determining when such constrained causal games admit a diagonal reduction, and when their almost equilibria have the compactness, semicontinuity, and recovery properties developed here, is a direction for further work.

\section*{Acknowledgements}

The financial support of the \emph{Centre Internacional de M\`etodes Num\`erics a l'Enginyeria} (CIMNE) of the \emph{Universitat Polit\`ecnica de Catalunya} (UPC), Spain, through the \emph{UNESCO Chair in Numerical Methods in Engineering} is gratefully acknowledged.


\begin{thebibliography}{99}

\bibitem{BrowderHess1972}
F.~E. Browder and P.~Hess,
\newblock Nonlinear mappings of monotone type in Banach spaces.
\newblock \emph{Journal of Functional Analysis} \textbf{11} (1972), no.~3, 251--294.
\newblock doi:10.1016/0022-1236(72)90070-5.

\bibitem{LerayLions1965}
J.~Leray and J.-L.~Lions,
\newblock Quelques r\'esultats de Vi\v{s}ik sur les probl\`emes elliptiques non lin\'eaires par les m\'ethodes de Minty--Browder.
\newblock \emph{Bulletin de la Soci\'et\'e Math\'ematique de France} \textbf{93} (1965), 97--107.
\newblock doi:10.24033/bsmf.1617.

\bibitem{Brezis1968}
H.~Br\'ezis,
\newblock \'Equations et in\'equations non lin\'eaires dans les espaces vectoriels en dualit\'e.
\newblock \emph{Annales de l'Institut Fourier} \textbf{18} (1968), no.~1, 115--175.
\newblock doi:10.5802/aif.280.

\bibitem{Zeidler1990}
E.~Zeidler,
\newblock \emph{Nonlinear Functional Analysis and Its Applications, II/B: Nonlinear Monotone Operators}.
\newblock Springer-Verlag, New York, 1990.
\newblock doi:10.1007/978-1-4612-0981-2.

\bibitem{Showalter1997}
R.~E.~Showalter,
\newblock \emph{Monotone Operators in Banach Space and Nonlinear Partial Differential Equations}.
\newblock Mathematical Surveys and Monographs, vol.~49, American Mathematical Society, Providence, RI, 1997.
\newblock doi:10.1090/surv/049.

\bibitem{Stackelberg1934}
H.~von Stackelberg,
\newblock \emph{Marktform und Gleichgewicht}.
\newblock Julius Springer, Wien and Berlin, 1934.

\bibitem{BasarOlsder1999}
T.~Ba\c{s}ar and G.~J.~Olsder,
\newblock \emph{Dynamic Noncooperative Game Theory}, 2nd ed.
\newblock Classics in Applied Mathematics, vol.~23, SIAM, Philadelphia, 1999.
\newblock doi:10.1137/1.9781611971132.

\bibitem{Minty1962}
G.~J.~Minty,
\newblock Monotone (nonlinear) operators in Hilbert space.
\newblock \emph{Duke Mathematical Journal} \textbf{29} (1962), no.~3, 341--346.
\newblock doi:10.1215/S0012-7094-62-02933-2.

\bibitem{Fan1961}
K.~Fan,
\newblock A generalization of Tychonoff's fixed point theorem.
\newblock \emph{Mathematische Annalen} \textbf{142} (1961), 305--310.
\newblock doi:10.1007/BF01353421.

\bibitem{BlumOettli1994}
E.~Blum and W.~Oettli,
\newblock From optimization and variational inequalities to equilibrium problems.
\newblock \emph{The Mathematics Student} \textbf{63} (1994), 123--145.

\bibitem{DeGiorgiMarinoTosques1980}
E.~De Giorgi, A.~Marino, and M.~Tosques,
\newblock Problemi di evoluzione in spazi metrici e curve di massima pendenza.
\newblock \emph{Atti della Accademia Nazionale dei Lincei. Classe di Scienze Fisiche, Matematiche e Naturali. Rendiconti} \textbf{68} (1980), no.~3, 180--187.

\bibitem{BrezisEkeland1976}
H.~Br\'ezis and I.~Ekeland,
\newblock Un principe variationnel associ\'e \`a certaines \'equations paraboliques.
\newblock \emph{C. R. Acad. Sci. Paris S\'er. A--B} \textbf{282} (1976), A971--A974 and A1197--A1198.

\bibitem{Nayroles1976}
B.~Nayroles,
\newblock Deux th\'eor\`emes de minimum pour certains syst\`emes dissipatifs.
\newblock \emph{C. R. Acad. Sci. Paris S\'er. A--B} \textbf{282} (1976), A1035--A1038.

\bibitem{AmbrosioDalMaso1990}
L.~Ambrosio and G.~Dal Maso,
\newblock A general chain rule for distributional derivatives.
\newblock \emph{Proceedings of the American Mathematical Society} \textbf{108} (1990), no.~3, 691--702.
\newblock doi:10.2307/2047789.

\bibitem{MielkeTheil2004}
A.~Mielke and F.~Theil,
\newblock On rate-independent hysteresis models.
\newblock \emph{Nonlinear Differential Equations and Applications} \textbf{11} (2004), 151--189.
\newblock doi:10.1007/s00030-003-1052-7.

\bibitem{MielkeRossiSavare2012}
A.~Mielke, R.~Rossi, and G.~Savar\'e,
\newblock $BV$ solutions and viscosity approximations of rate-independent systems.
\newblock \emph{ESAIM: Control, Optimisation and Calculus of Variations} \textbf{18} (2012), no.~1, 36--80.
\newblock doi:10.1051/cocv/2010054.

\bibitem{MielkeRossiSavare2016}
A.~Mielke, R.~Rossi, and G.~Savar\'e,
\newblock Balanced Viscosity ($BV$) solutions to infinite-dimensional rate-independent systems.
\newblock \emph{Journal of the European Mathematical Society} \textbf{18} (2016), no.~9, 2107--2165.
\newblock doi:10.4171/JEMS/639.

\bibitem{DiestelUhl1977}
J.~Diestel and J.~J.~Uhl, Jr.,
\newblock \emph{Vector Measures}.
\newblock Mathematical Surveys, vol.~15, American Mathematical Society, Providence, RI, 1977.
\newblock doi:10.1090/surv/015.

\bibitem{OrtizRepetto1999}
M.~Ortiz and E.~A.~Repetto,
\newblock Nonconvex energy minimization and dislocation structures in ductile single crystals.
\newblock \emph{Journal of the Mechanics and Physics of Solids} \textbf{47} (1999), 397--462.
\newblock doi:10.1016/S0022-5096(97)00096-3.

\bibitem{ContiOrtiz2005}
S.~Conti and M.~Ortiz,
\newblock Dislocation microstructures and the effective behavior of single crystals.
\newblock \emph{Archive for Rational Mechanics and Analysis} \textbf{176} (2005), 103--147.
\newblock doi:10.1007/s00205-004-0353-2.

\bibitem{ContiHauretOrtiz2007}
S.~Conti, P.~Hauret, and M.~Ortiz,
\newblock Concurrent multiscale computing of deformation microstructure by relaxation and local enrichment with application to single-crystal plasticity.
\newblock \emph{Multiscale Modeling \& Simulation} \textbf{6} (2007), no.~1, 135--157.
\newblock doi:10.1137/060662332.

\bibitem{ContiOrtiz2008}
S.~Conti and M.~Ortiz,
\newblock Minimum principles for the trajectories of systems governed by rate problems.
\newblock \emph{Journal of the Mechanics and Physics of Solids} \textbf{56} (2008), no.~5, 1885--1904.
\newblock doi:10.1016/j.jmps.2007.11.006.

\bibitem{McCormick1988}
P.~G.~McCormick,
\newblock Theory of flow localisation due to dynamic strain ageing.
\newblock \emph{Acta Metallurgica} \textbf{36} (1988), no.~12, 3061--3067.
\newblock doi:10.1016/0001-6160(88)90043-0.

\bibitem{Ananthakrishna2007}
G.~Ananthakrishna,
\newblock Current theoretical approaches to collective behavior of dislocations.
\newblock \emph{Physics Reports} \textbf{440} (2007), no.~4--6, 113--259.
\newblock doi:10.1016/j.physrep.2006.10.003.

\bibitem{Yilmaz2011}
A.~Yilmaz,
\newblock The Portevin--Le Ch\^atelier effect: a review of experimental findings.
\newblock \emph{Science and Technology of Advanced Materials} \textbf{12} (2011), 063001.
\newblock doi:10.1088/1468-6996/12/6/063001.

\bibitem{DivouxEtAl2016}
T.~Divoux, M.~A.~Fardin, S.~Manneville, and S.~Lerouge,
\newblock Shear banding of complex fluids.
\newblock \emph{Annual Review of Fluid Mechanics} \textbf{48} (2016), 81--103.
\newblock doi:10.1146/annurev-fluid-122414-034416.

\bibitem{OsovskiRittelVenkert2013}
S.~Osovski, D.~Rittel, and A.~Venkert,
\newblock The respective influence of microstructural and thermal softening on adiabatic shear localization.
\newblock \emph{Mechanics of Materials} \textbf{56} (2013), 11--22.
\newblock doi:10.1016/j.mechmat.2012.09.008.

\bibitem{BaiCogswellBazant2011}
P.~Bai, D.~A.~Cogswell, and M.~Z.~Bazant,
\newblock Suppression of phase separation in LiFePO$_4$ nanoparticles during battery discharge.
\newblock \emph{Nano Letters} \textbf{11} (2011), no.~11, 4890--4896.
\newblock doi:10.1021/nl202764f.

\bibitem{Bazant2017}
M.~Z.~Bazant,
\newblock Thermodynamic stability of driven open systems and control of phase separation by electro-autocatalysis.
\newblock \emph{Faraday Discussions} \textbf{199} (2017), 423--463.
\newblock doi:10.1039/C7FD00037E.

\bibitem{Bhattacharya2003}
K.~Bhattacharya,
\newblock \emph{Microstructure of Martensite: Why It Forms and How It Gives Rise to the Shape-Memory Effect}.
\newblock Oxford Series on Materials Modelling, Oxford University Press, Oxford, 2003.
\newblock ISBN 978-0-19-850934-9; doi:10.1093/oso/9780198509349.001.0001.

\bibitem{DeGiorgi1993}
E.~De Giorgi,
\newblock New problems on minimizing movements.
\newblock In \emph{Boundary Value Problems for Partial Differential Equations and Applications}, C.~Baiocchi and J.-L.~Lions, eds., RMA Res. Notes Appl. Math., vol.~29, Masson, Paris, 1993, pp.~81--98.

\bibitem{Ambrosio1995}
L.~Ambrosio,
\newblock Minimizing movements.
\newblock \emph{Rend. Accad. Naz. Sci. XL Mem. Mat. Appl.}, ser.~5, \textbf{19} (1995), 191--246.

\bibitem{AmbrosioGigliSavare2008}
L.~Ambrosio, N.~Gigli, and G.~Savar\'e,
\newblock \emph{Gradient Flows: In Metric Spaces and in the Space of Probability Measures}, 2nd ed.
\newblock Lectures in Mathematics ETH Z\"urich, Birkh\"auser, Basel, 2008.
\newblock doi:10.1007/978-3-7643-8722-8.

\bibitem{Reshetnyak1968}
Yu.~G.~Reshetnyak,
\newblock Weak convergence of completely additive vector functions on a set.
\newblock \emph{Siberian Mathematical Journal} \textbf{9} (1968), 1039--1045.
\newblock doi:10.1007/BF02196453.

\bibitem{Rockafellar1971IntegralII}
R.~T.~Rockafellar,
\newblock Integrals which are convex functionals. II.
\newblock \emph{Pacific Journal of Mathematics} \textbf{39} (1971), no.~2, 439--469.
\newblock doi:10.2140/pjm.1971.39.439.

\bibitem{Rockafellar1971ConvexIntegralDuality}
R.~T.~Rockafellar,
\newblock Convex integral functionals and duality.
\newblock In \emph{Contributions to Nonlinear Functional Analysis}, E.~H.~Zarantonello, ed., Academic Press, New York, 1971, pp.~215--236.
\newblock doi:10.1016/B978-0-12-775850-3.50012-1.

\bibitem{AFP00}
L.~Ambrosio, N.~Fusco, and D.~Pallara,
\newblock \emph{Functions of Bounded Variation and Free Discontinuity Problems}.
\newblock Oxford Mathematical Monographs, The Clarendon Press, Oxford University Press, New York, 2000.
\newblock doi:10.1093/oso/9780198502456.001.0001.

\bibitem{Brezis2011}
H.~Br\'ezis,
\newblock \emph{Functional Analysis, Sobolev Spaces and Partial Differential Equations}.
\newblock Universitext, Springer, New York, 2011.
\newblock doi:10.1007/978-0-387-70914-7.

\bibitem{MielkeOrtiz2008}
A.~Mielke and M.~Ortiz,
\newblock A class of minimum principles for characterizing the trajectories and the relaxation of dissipative systems.
\newblock \emph{ESAIM: Control, Optimisation and Calculus of Variations} \textbf{14} (2008), no.~3, 494--516.
\newblock doi:10.1051/cocv:2007064.

\bibitem{MielkeStefanelli2011}
A.~Mielke and U.~Stefanelli,
\newblock Weighted energy--dissipation functionals for gradient flows.
\newblock \emph{ESAIM: Control, Optimisation and Calculus of Variations} \textbf{17} (2011), no.~1, 52--85.
\newblock doi:10.1051/cocv/2009043.

\bibitem{BenamouBrenier2000}
J.-D.~Benamou and Y.~Brenier,
\newblock A computational fluid mechanics solution to the Monge--Kantorovich mass transfer problem.
\newblock \emph{Numerische Mathematik} \textbf{84} (2000), no.~3, 375--393.
\newblock doi:10.1007/s002110050002.

\bibitem{Negri2014}
M.~Negri,
\newblock Quasi-static rate-independent evolutions: characterization, existence, approximation and application to fracture mechanics.
\newblock \emph{ESAIM: Control, Optimisation and Calculus of Variations} \textbf{20} (2014), no.~4, 983--1008.
\newblock doi:10.1051/cocv/2014004.

\bibitem{AliprantisBorder2006}
C.~D. Aliprantis and K.~C. Border,
\newblock \emph{Infinite Dimensional Analysis: A Hitchhiker's Guide},
\newblock 3rd ed., Springer, Berlin, 2006.

\bibitem{Warga1972}
J.~Warga,
\newblock \emph{Optimal Control of Differential and Functional Equations},
\newblock Academic Press, New York, 1972.

\bibitem{CastaingRaynaudDeFitteValadier2004}
C.~Castaing, P.~Raynaud de Fitte, and M.~Valadier,
\emph{Young Measures on Topological Spaces: With Applications in Control Theory and Probability Theory}.
Mathematics and Its Applications, vol.~571, Kluwer Academic Publishers,
Dordrecht, 2004.
doi:10.1007/1-4020-1964-5.

\end{thebibliography}
\end{document}